\documentclass[12pt, times]{amsart}

\usepackage{amsthm}
\usepackage{color}
\usepackage{amssymb, amsmath, amsfonts, latexsym}
\usepackage{enumerate}

\usepackage{hyperref}

\allowdisplaybreaks[2]
\usepackage{booktabs}
\usepackage[table,xcdraw]{xcolor}
\usepackage{longtable}

\usepackage{tikz}
\usepackage{pgfplots}
\usepackage{algorithm,algpseudocode,float}
\usepackage{lipsum}

\makeatletter
\newenvironment{breakablealgorithm}
  {% \begin{breakablealgorithm}
     \refstepcounter{algorithm}% New algorithm
     \hrule height 0.8pt depth0pt \kern2pt% \@fs@pre for \@fs@ruled
     \renewcommand{\caption}[2][\relax]{% Make a new \caption
       {\raggedright\textbf{\ALG@name~\thealgorithm} ##2\par}%
       \ifx\relax##1\relax % #1 is \relax
         \addcontentsline{loa}{algorithm}{\protect\numberline{\thealgorithm}##2}%
       \else % #1 is not \relax
         \addcontentsline{loa}{algorithm}{\protect\numberline{\thealgorithm}##1}%
       \fi
       \kern2pt\hrule  \kern2pt
     }
  }{
     \kern2pt\hrule\relax% \@fs@post for \@fs@ruled
  }
\makeatother
\usepackage{bbm}
 \usepackage{fancyhdr}
\newtheorem{thm}{Theorem}[]
\newtheorem{lem}{Lemma}[]
\newtheorem{cor}[thm]{Corollary}

\newtheorem{rmk}[thm]{Remark}
\theoremstyle{definition}

\numberwithin{equation}{section} 

\author[Ma, Y.]{Yutao Ma}
\title[ISE for the Extremal Eigenvalue]{\bf Importance Sampling for the Extremal Eigenvalue of $\beta$-Jacobi ensemble}  

\date{}

\address{Yutao MA\\ School of Mathematical Sciences $\&$ Laboratory  of Mathematics and Complex Systems of Ministry of Education, Beijing Normal University, 100875 Beijing, China.} 
\thanks{The research of Yutao Ma was supported in part by NSFC 12171038, 11871008 and 985 Projects.}
\email{mayt@bnu.edu.cn}

\author[Wang, S.]{Siyu Wang*}
\address{Siyu Wang\\ School of Mathematical Sciences $\&$ Laboratory of Mathematics and Complex Systems of Ministry of Education, Beijing Normal University, 100875 Beijing, China.}
\email{wang\_siyu@mail.bnu.edu.cn}

\begin{document}

\maketitle
\begin{abstract} 
This paper focuses on rare events associated with the tail probabilities of the extremal eigenvalues in the $\beta$-Jacobi ensemble, which plays a critical role in both multivariate statistical analysis and statistical physics.
Under the ultra-high dimensional setting, we give an exact approximation for the tail probabilities and construct an efficient estimator for the tail probabilities.
Additionally, we conduct a numerical study to evaluate the practical performance of our algorithms. The simulation results demonstrate that our method offers an efficient and accurate approach for evaluating tail probabilities in practice.
\end{abstract} 

\vskip0.2cm

{\bf AMS Classification Subjects 2020:} 60F10, 	60B20

\section{Introduction}
In classical multivariate analysis, the tail probabilities of the largest eigenvalues play a crucial role and are widely applied to various statistical problems, such as linear regression, MANOVA (Multivariate Analysis of Variance), canonical correlation analysis, and the equality of covariance matrices.
This paper focuses on the tail probabilities of the largest eigenvalues of the $\beta$-Jacobi ensemble, whose density function is given in equation \eqref{defjacobi}. 

Specifically, when $\beta = 1$, the $\beta$-Jacobi ensemble corresponds to the eigenvalue distribution of real MANOVA matrices (\cite{Muirhead}). 
Consider an $n \times p_1$ normal data matrix $X,$ where each row is an independent observation from $\mathcal{N}(0, \Sigma).$ The matrix $A = X^{\prime}X$ follows a Wishart distribution, $A \sim W_n(\Sigma, p_1),$ and let $B \sim W_n(\Sigma, p_2),$ independent of $A.$ Then, the distribution of  MANOVA matrix $(A+B)^{-1}B$ follows equation \eqref{defjacobi}(with $\beta =1$). When the elements of $X$ are complex or quaternion, the cases $\beta = 2$ and $\beta = 4$ correspond to the eigenvalue distributions of complex and quaternion MANOVA matrices, respectively. The distribution of the largest eigenvalue of the MANOVA matrix has broad applications in multivariate statistical analysis. For instance, Johnstone's work (\cite{Johnstone08}) on the convergence of the largest eigenvalue has been used to address several classic problems in multivariate statistics. In canonical correlation analysis, the largest eigenvalue quantifies the linear relationship between two sets of variables; in the analysis of angles and distances between subspaces, it assesses the similarity between high-dimensional datasets; in multivariate linear models, it tests the significance of linear hypotheses; in testing the equality of covariance matrices, the largest eigenvalue is used to compare the covariance structures across different groups; and in multiple discriminant analysis, it measures the separability between categories, thereby improving the accuracy of group classification.

For general $\beta > 0,$ the $\beta$-Jacobi ensemble can be viewed as a non-equilibrium system in statistical physics, where $\beta$ serves as the inverse temperature, controlling the interaction strength between eigenvalues. In the study of energy spectrum distributions and phase transition behavior, the largest eigenvalue of the $\beta$-Jacobi ensemble plays a crucial role in describing the system's extreme behavior in different energy states, while also revealing the characteristics of critical points during phase transitions.

In the study of the aforementioned multivariate statistical analyses and statistical physics models, sampling from the event $\{\lambda_{\text{max}} > x\}$ is often required. 
For general $\beta > 0,$ Killip and Nenciu (\cite{Killip}) constructed a tridiagonal random matrix model for the $\beta$-Jacobi ensemble. 
Using this model, the eigenvalues of the tridiagonal matrix can be computed via Monte Carlo methods to simulate the behavior of $\lambda_{\text{max}}.$ 
However, when the probability of the event $\{\lambda_{\text{max}} > x\}$ is extremely small or the data dimension is very large, direct Monte Carlo simulations become highly inefficient. These simulations require substantially larger sample sizes, resulting in a significant increase in computational cost.

Large deviation theory provides a valuable tool for quantifying the probabilities of rare events (\cite{Shwartz}). However, practical applications face challenges, as approximations based on large deviation results may lack direct applicability, leading to insufficient accuracy in estimating the probabilities of rare events. In this context, importance sampling emerges as the preferred approach. Unlike classical Monte Carlo methods, importance sampling significantly reduces simulation time and improves efficiency. Essentially, importance sampling is a modified Monte Carlo technique that produces estimates with a smaller coefficient of variation--a natural performance metric for rare event estimation.

The focus of this paper is to construct an effective estimator for the tail probabilities of the $\beta$-Jacobi ensemble using importance sampling.
Asmussen and Kroese (\cite{Asmussen2006}) addressed the problem of tail probability estimation for independent and identically distributed heavy-tailed random variables, developing importance sampling estimators through a two-step construction.   
Building on a similar structural approach, Jiang {\it et al.} (\cite{JiangRare}) delved into the asymptotical behavior of top eigenvalues for the $\beta$-Laguerre ensemble,  using the tridiagonal random matrix characterization of the $\beta$-Laguerre ensemble and the ``three-step peeling" technique to precisely estimate the tail probability. 
Inspired by Jiang {\it et al.} (\cite{JiangRare}), we previously constructed an asymptotically effective estimator for the tail probabilities of the $\beta$-Jacobi ensemble in \cite{MW2023}. In this paper, we take this further by constructing a strongly efficient estimator for the tail probabilities.

A $\beta$-Jacobi ensemble $\mathcal J_n(p_1, p_2) $ is a set of random variables $\lambda:=\left(\lambda_1, \lambda_2, \cdots, \lambda_n\right) \in[0,1]^n$ with joint probability density function
\begin{align}\label{defjacobi} f_n^{p_1, p_2}\left(x_1, \cdots, x_n\right)=C_n^{p_1, p_2} \prod_{1 \leq i<j \leq n}\left|x_i-x_j\right|^\beta \prod_{i=1}^n x_i^{r_{1, n}-1}\left(1-x_i\right)^{r_{2, n}-1},
 \end{align}
where $p_1, p_2 \geq n, r_{i, n}:=\frac{\beta\left(p_i-n+1\right)}{2}$ for $i=1,2$ and $\beta>0$, and the normalizing constant $C_n^{p_1, p_2}$ is given by
$$
C_n^{p_1, p_2}=\prod_{j=1}^n \frac{\Gamma\left(1+\frac{\beta}{2}\right) \Gamma\left(\frac{\beta(p-n+j}{2}\right)}{\Gamma\left(1+\frac{\beta j}{2}\right) \Gamma\left(\frac{\beta\left(p_1-n+j\right)}{2}\right) \Gamma\left(\frac{\beta\left(p_2-n+j\right)}{2}\right)}
$$
with $p=p_1+p_2$ and $\Gamma(a)=\int_0^{\infty} x^{a-1} e^{-x} dx$ with $a>0.$ 
Let $\lambda_{(1)}< \lambda_{(2)}< \cdots < \lambda_{(n)}$ be the order statistics of the $\beta$-Jacobi ensemble $\mathcal{J}_n(p_1,p_2)$, with joint distribution $G_{n}^{p_1,p_2}$ and joint density function
\begin{align}\label{od}
	n ! f_n^{p_1, p_2}\left(x_1, \cdots, x_n\right) \mathbf 1_{\left(x_1<\cdots<x_n\right)}. 
\end{align}  

In this paper, we provide asymptotic approximations for the tail probabilities $\mathbb P\left( \lambda_{(n)} > \frac{p_1}{p}x\right)$  of the extremal eigenvalues under certain conditions on $\beta$, $n$, $p_1$, and $p_2$.
Moreover, we construct efficient Monte Carlo simulation algorithms to compute the tail probabilities.
In the following sections, we introduce the fundamental setup for importance sampling and present the main results of our paper. 

 \section{Tail Probability Estimation}\label{mainresult}

\subsection{Basic importance sampling}

Consider a random variable $Y$ with probability distribution $Q$ taking values on spaces $\Omega$, and a real-valued function $\omega$ on $\Omega$. 
Importance sampling is based on the basic identity, for fixed $B \in \mathcal B,$
$$
	 \mathbb E\left[\mathbf 1_{\left\{\omega\left(Y\right) \in B \right\}}\right]
	 = \mathbb E^R\left[\mathbf 1_{\left\{\omega\left(Y\right) \in B \right\}} \frac{\mathrm{d} Q}{\mathrm{d} R}\right],  
$$
where $R$ is a probability measure such that the Radon-Nikodym derivative $\mathrm{d} Q / \mathrm{d} R$ is well defined on the set $\left\{\omega\left(Y\right) \in B \right\}$. 
Here, $\mathcal B$ denotes the Borel $\sigma$-algebra on $\mathbb R$, and $\mathbb E$ or $\mathbb E^R$ denotes the expectations under the measures $Q$ or $R$, respectively. Then, the random variable $F$ is given by 
\begin{align}\label{deff} 
	F=\frac{\mathrm{d} Q}{\mathrm{d} R}\mathbf 1_{\left\{\omega\left(Y\right) \in B\right\}} 
\end{align} 
is an unbiased estimator of $\mathbb P\left( \omega\left(Y\right) \in B \right)$ under the measure $R$. An averaged importance sampling estimator based on the measure $Q$ is obtained by simulating $N$ i.i.d. copies $F^{(1)}, \ldots, F^{(N)}$ of $F$ under $R$ and computing the empirical average $\widehat{F}_N=\left(F^{(1)}+\cdots+F^{(N)}\right) / N$. 
Moreover, its variance is
$$
\mathrm{Var}^R\left( \widehat{F}_N \right) = \frac{\mathbb E^R \left[F^2\right] - \mathbb P\left(\omega\left(Y\right) \in B \right)^2 }{N},
$$
which decreases as $N$ increases. The coefficient of variation is given by
\begin{align}\label{errorvar}
	 \mathrm{C. O. V.} \left(\widehat{F}_N\right) :=	\frac{\sqrt{\mathrm{Var}^R\left( \widehat{F}_N \right)}}{\mathbb E^R\left( \widehat{F}_N \right) }  = \frac{1}{\sqrt{N}} \left(\frac{\mathbb E^R \left[F^2\right]}{\mathbb P\left(\omega\left(Y\right) \in B \right)^2}-1\right)^{1/2}.
\end{align}  
A central issue is to optimize the left-hand side of  \eqref{errorvar} via selecting $R.$  
In the case that $R$ is taken as $$R^*(\cdot):=\mathbb P\left(\cdot \mid \omega\left(Y\right) \in B\right),$$ it turns out that $${\rm Var}^{R}(\widehat{F}_N)=0.$$ Unfortunately, the optimal IS density $R^*$ is computationally unavailable since we do not know the value of $\mathbb P\left( \omega\left(Y\right) \in B \right).$

Given a family of Borel sets $(B_n; n\geq 1)$ such that $$0<\mathbb P\left( \omega\left(Y\right) \in B_n \right) \text{ and } \lim_{n\to\infty}\mathbb{P}\left( \omega\left(Y\right) \in B_n \right)=0,$$ it is necessary to consider the quotient $$\frac{\mathbb{E}^R \left[F_n^2\right]}{\mathbb P\left(\omega\left(Y\right) \in B \right)^2}$$ when constructing a sequence of estimators $(F_n; n\geq 1)$ for $\mathbb P\left( \omega\left(Y\right) \in B_n \right)$.  In the context of rare-event simulations (e.g., Asmussen and Glynn, \cite{Asmussen}), they say that $F_n$ is strongly efficient if
\begin{align}\label{defstrongeff}
		\varlimsup _{n \rightarrow \infty} \frac{\mathbb E^R \left[F_n^2\right]}{\mathbb P\left(\omega\left(Y\right) \in B_n\right)^2}<\infty .
\end{align}  

Another efficiency concept slightly weaker than strongly efficient is asymptotically (or logarithmically) efficient: $\mathbb E^R \left[F_n^2\right]$ tends to $0$ so fast that
$$
	\lim_{n \rightarrow \infty} \frac{\left| \log \mathbb E^R \left[F_n^2\right]\right|}{2 \left|\log \mathbb P\left(\omega\left(Y\right) \in B_n \right)\right|} = 1.
$$
For a detailed discussion on efficiency in rare-event simulations, see Asmussen and Glynn \cite{Asmussen}, Bucklew \cite{Bucklew}, or Juneja and Shahabuddin \cite{Juneja}.

 Constructing efficient estimators for the tail probabilities of the $\beta$-Jacobi ensemble simplifies the computation of extreme value probabilities in applications across multivariate statistical analysis and statistical physics. Strongly efficient estimators offer greater stability, consistency, and accuracy, though they often require more stringent parameter conditions. In contrast, asymptotically efficient estimators typically have lower computational costs, especially in large sample scenarios, making them a more practical choice when computational resources are limited.

\subsection{Importance sampling estimator for $p \lambda_{(n)}/p_1$}\label{ISE}
Next, we construct the IS estimator $F_n$ for  $\mathbb P\left( \lambda_{(n)} > \frac{p_1}{p}x\right)$ when $x>1$.  
Consider the order statistics $\left(\widetilde{\lambda}_{(1)}, \cdots, \widetilde{\lambda}_{(n-1)}\right)$ of the $\beta$-Jacobi ensemble $\mathcal J_{n-1}(p_1-1, p_2-1)$.
 Let $\widetilde \lambda_{n}$ be a random variable with conditional density function \begin{equation}\label{defh}
	h\left(y\right):= \frac{r}{1- e^{ -r(1-p_1x/p\vee \widetilde{\lambda}_{(n-1)}) }}  e^{-r\left(y-\left( p_1x /p \right) \vee \widetilde{\lambda}_{(n-1)}  \right)} \mathbf 1_{\{  \left( p_1x /p \right) \vee \widetilde{\lambda}_{(n-1)} < y< 1 \}}, 
\end{equation} given $\left( \widetilde{\lambda}_{(1)}, \cdots, \widetilde{\lambda}_{(n-1)}\right),$ where the rate $r  = \frac{\beta p(x-1)}{2x}.$  Let $L_n^{p_1, p_2}$ be the joint distribution of $\left(\widetilde{\lambda}_{(1)}, \cdots, \widetilde{\lambda}_{(n-1)}, \widetilde \lambda_{n} \right).$ The IS estimator $F_n$ for $\mathbb P\left( \lambda_{(n)} > \frac{p_1}{p}x\right)$ is then given by $$F_n  = \frac{\mathrm d G_n^{p_1, p_2}}{\mathrm d L_n^{p_1, p_2}} (\lambda_{(1)}, \cdots, \lambda_{(n)})\mathbf 1_{\left\{ \lambda_{(n)} > \frac{p_1}{p}x\right\}}.$$

The following theorem  show the efficiency of  $F_n$ 
under the following different assumptions 
\begin{align*} 
&{\bf H}_0:	\quad  \beta n p_1^2\ll p_2, \quad \frac{\log n}{n}\vee \frac{n}{p_1}\ll \beta  \quad \text{and}  \quad   (\beta n)^5\ll (\beta p_1)^3;\\
&{\bf H}_1:	\quad  \beta n p_1^2=O(p_2), \quad \frac{\log n}{n}\vee \frac{n}{p_1}\ll \beta  \quad \text{and}  \quad (\beta n)^5=O((\beta p_1)^3);  
\\
&{\bf H}_2:	\quad   np_1\ll p_2, \quad \frac{\log n}{n}\vee \frac{n}{p_1}\ll \beta  \quad \text{and}   \quad n\ll p_1.
\end{align*}  
Here, the notation $x\ll y$ means 
$\lim\limits_{y\to\infty}\frac{x}{y}=0$ and $x=O(y)$ represents 
$\lim\limits_{x\to\infty}\frac{x}{y}=c\neq 0.$

\begin{thm}\label{thmsimu}
Given $x>1$ fixed. Let $F_n$ be constructed as above. We have the following assertions. 
\begin{enumerate} 
\item Suppose that the assumption {\bf H$_2$} holds. The IS estimator $F_n$ is asymptotically efficient. 
\item  The estimator $F_n$ is strongly efficient in the case when {\bf H}$_1$ is true. 
\item Under the assumption {\bf H}$_0$, 
$F_n$ satisfies 
\begin{align}
\lim_{n \to \infty} \frac{\mathbb E^L\left[ F_n^2 \right]}{\mathbb P\left(\lambda_{(n)} > \frac{p_1}{p}x \right)^2} = 1. \label{appro1}
\end{align}  
This indicates that
$$
\lim _{n \rightarrow \infty} \sup _{A \in \mathcal{F}}\left|L_n^{p_1, p_2}(A)-\mathbb{P}\left(A \left| \lambda_{(n)}>\frac{p_1}{p} x \right.\right)\right| = 0,
$$
where $\mathcal F$ is the $\sigma$-field $\mathcal B\left([0,1]^n\right)$. Here, $\mathbb E^L$  represents the expectations with respect to the measure $L_n^{p_1, p_2}.$ 
\end{enumerate} 
\end{thm} 

\begin{rmk}
In importance sampling, it is crucial to select an appropriate measure to minimize the estimator's coefficient of variation. The optimal IS measure, denoted by
$$ L^*(\cdot):= \mathbb{P}\left( \cdot \left| \lambda_{(n)}>\frac{p_1}{p} x \right.\right),$$ can theoretically achieve a coefficient of variation of zero, but it is impractical in most cases. Theorem \ref{thmsimu} shows that $L_n^{p_1, p_2}$ converges to $L^*$ in the sense of total variation distance.
\end{rmk}

\begin{rmk}
The IS estimator $F_n$ for $\mathbb P\left( \lambda_{(n)} > \frac{p_1}{p}x\right)$ constructed in \cite{MW2023} is almost the same except that $0<r<\frac{\beta p(x-1)}{2x}.$  Precisely, in \cite{MW2023}, we verified that $F_n$ is asymptotically efficient under a weaker condition \textbf{A}$_1$ compared to \textbf{H}$_2$,  which is  $${\mathbf A}_1: \quad\quad \lim_{n\to \infty} \frac{n}{p_1} = 0, \quad  \lim_{n\to \infty} \frac{p_1}{p_2} = \sigma \in [0, \infty) \quad \text{ and } \quad  \log n\ll \beta n.$$
Both the condition {\bf A}$_1$ and the constraint that $r$ can not be the upper bound  $\frac{\beta p(x-1)}{2x}$   come from the requirement of the large deviation for $p\lambda_{(n)}/p_1.$ In this paper, we use another way to estimate the corresponding terms, which optimizes $r$ to be $\frac{\beta p(x-1)}{2x}$ while somehow  enhances the conditions on the parameters $\beta, p_1, p_2$ and $n.$ 
\end{rmk}

In the proof of Theorem \ref{thmsimu}, we can give  an exact approximation for $\mathbb P(\lambda_{(n)} > \frac{p_1}{p}x).$

\begin{cor} Under the assumption ${\bf H}_0$,
\begin{align*}
	 \mathbb P(\lambda_{(n)} > \frac{p_1}{p}x ) 
	= & \frac{2 n A_n^{p_1, p_2} }{\beta p}\left(\frac{p_1(x-1)}{p}\right)^{\beta(n-1)}u_n\left( \frac{p_1 x}{p} \right)\times \exp\left\{-\frac{\alpha_1 \beta n^2}{p_1(x-1)^2}+o\left(1\right) \right\},
\end{align*}  
where $u_n$ is defined in equation \eqref{defun}.
\end{cor}
In Section \ref{Smainthm}, 
we showed the `$\geq$' case. The `$\leq$'  case follows similarly, though we omit the proof here for brevity.

\begin{rmk}[Smallest eigenvalue] If $\mu = (\mu_1, ..., \mu_n)$ is from $\beta$-Jacobi ensemble $\mathcal J_n(p_2, p_1)$, let $\mu_{(1)}$ denote the smallest eigenvalue. Its distribution is given by
$$
\mu_{(1)} \overset{\mathcal D}{=} 1- \lambda_{(n)}.
$$ In particular,  Theorem \ref{thmsimu} will give a useful approximation to the lower tail of $\mu_{(1)}.$
	
\end{rmk}

The rest of the paper is organized as follows. 
In Section \ref{Snums}, we conduct a numerical study to evaluate the practical performance of our algorithms.
Before proving the Theorem \ref{thmsimu}, some lemmas needed are listed in Section \ref{Stechlem}. 
We will prove Theorem \ref{thmsimu} in Sections \ref{Smainthm}. And the proofs of lemmas in the previous sections will be presented in the supplementary material.

\section{Numerical study}\label{Snums}

In this section, we perform a numerical study to evaluate the effectiveness of IS estimator $$F_n  = \frac{\mathrm d G_n^{p_1, p_2}}{\mathrm d L_n^{p_1, p_2}} (\lambda_{(1)}, \cdots, \lambda_{(n)})\mathbf 1_{\left\{ \lambda_{(n)} > \frac{p_1}{p}x\right\}},$$ as defined in Theorem \ref{thmsimu}.
Note that the measure $L_n^{p_1, p_2}$ has a density function of $$(n-1)! f_{n-1}^{p_1-1, p_2-1}\left(x_1, \cdots, x_{n-1}\right) h(x_n),$$ where $h$ is given in \eqref{defh}.  
Based on the method in \cite{JiangRare}, we characterize the measure $L_n^{p_1,p_2}$ in Algorithm \ref{algoL}.

\vspace{1em}

\begin{breakablealgorithm}
\caption{Sampling under the measure $L_n^{p_1,p_2}$ }\label{algoL}
\begin{algorithmic}[1]
	\Require the parameters $\beta, n, p_1, p_2, x$ in $ \mathbb{P}\left(\lambda_{(n)}>  \frac{p_1x}{p} \right) $
	
	\Ensure $n$ dimensional random vector following the measure $L_n^{p_1,p_2}$ 
	
	\Function{FunJ}{$\beta, n, p_1, p_2$}
	\State{$ \mathcal{J}_n^{p_1, p_2} \gets $ Zero Matrix of $n \times n$}
	\For{ $k$ from $1$ to $n$ }
		\State{ $c(k) \gets$ sample from $\operatorname{Beta}\left( \frac{\beta}{2} \left(p_1-k+1  \right), \frac{\beta}{2} \left(p_2-k+1 \right) \right)$}
		\State{ $s(k) \gets $ sample from $\operatorname{Beta}\left( \frac{\beta}{2} \left(n-k \right),  \frac{\beta}{2} \left(p_1+p_2-n-k+1  \right)  \right)$ }
	\EndFor
	\State{$(c(0), s(0)) \gets (0, 0)$}
	
	\For{ $k$ from $1$ to $n$ }
		\State{ $\mathcal{J}_n^{p_1, p_2}(k, k) \gets s(k-1)(1-c(k-1)) +c(k)(1-s(k-1)) $  }
		\While{$k \neq n$} 
		\State{ $ \mathcal{J}_n^{p_1, p_2}(k, k+1) \gets \sqrt{c(k)(1-c(k))s(k)(1-s(k-1))}$}
		\State{ $\mathcal{J}_n^{p_1, p_2}(k+1,k ) \gets \mathcal{J}_n^{p_1, p_2}(k, k+1) $ }
		\EndWhile
	\EndFor
	\State{\textbf{return} $\mathcal{J}_n^{p_1, p_2}$}	
	\EndFunction
	
	\State{ $\left( \lambda_{(1)} , \ldots, \lambda_{(n-1)} \right)  \gets $ the order statistics of FunJ$(\beta, n-1, p_1-1, p_2-1)$ eigenvalues}
	
	\Repeat 
	\State{$\tilde \lambda_{n} \gets$ sample from $\operatorname{Exp}\left( \frac{\beta p(x-1)}{2x} \right)$}	
	\State{$\tilde \lambda_{n} \gets \tilde \lambda_{n} + \frac{p_1x }{p} \vee \lambda_{(n-1)}$}
	\Until{($\tilde \lambda_{n} <1 $)}
	
	\State{\textbf{return} $\left(\lambda_{(1)}, \cdots, \lambda_{(n-1)}, \tilde \lambda_{n} \right)$}
\end{algorithmic}
\end{breakablealgorithm}

\vspace{1em}
From \cite{Killip}, we know the eigenvalues of   FunJ$(\beta, n-1, p_1-1, p_2-1)$  has joint density function $f_{n-1}^{p_1-1, p_2-1}$ as defined in \eqref{defjacobi}.  It is easy to show that  the sampled $\left(\lambda_{(1)}, \cdots, \lambda_{(n-1)}, \tilde \lambda_{n} \right)$ follow the distribution $L_n^{p_1, p_2}$ as proposed in Theorem \ref{thmsimu}.  Let's give a brief expression for the IS estimator $F_n$. Note that the  IS estimator $F_n$ is 
 \begin{align*}
	F_n = \frac{n ! f_n^{p_1, p_2}\left(\lambda_{(1)}, \cdots, \lambda_{(n)}\right) }{(n-1)! f_{n-1}^{p_1-1, p_2-1}\left(\lambda_{(1)}, \cdots, \lambda_{(n-1)}\right) h(\lambda_{(n)}) }\mathbf{1}_{\left\{ \lambda_{(n)} > p_1 x/p  \right\}}.
\end{align*}  
Using the exact form of $f_{n}^{p_1, p_2}$ to write the IS estimator $F_n$ as 
\begin{align*}
	F_n =  n A_n \prod_{i=1}^{n-1} \left( \lambda_{(n)}- \lambda_{(i)} \right)^{\beta} \lambda_{(n)}^{r_{1,n}-1} \left(1-\lambda_{(n)}\right)^{r_{2,n}-1} h^{-1}\left(\lambda_{(n)}\right) \mathbf{1}_{\left\{ \lambda_{(n)} > p_1 x/p  \right\}},
\end{align*}  
where  $A_n^{p_1, p_2}$ is defined by \eqref{defAn}.

In the following algorithm, we evaluate the actual performance of the IS estimator $F_n$ for $ \mathbb{P}\left(\lambda_{(n)}>  \frac{p_1x}{p} \right) $ by calculating the mean and coefficient of variation of $N$ independent samples $F^{(1)}, \ldots, F^{(N)}.$ 

\vspace{1em}
\begin{breakablealgorithm}
\caption{Some statistical indicators of $F_n$}\label{algoF}
\begin{algorithmic}[1]
	\Require the parameters $\beta, n, p_1, p_2, x$ in $ \mathbb{P}\left(\lambda_{(n)}>  \frac{p_1x}{p} \right) $ and number of simulations $N$.
 	\Ensure The mean Est., standard deviation Std., and coefficient of variation C.O.V. of $N$ independent samples $F^{(1)}, \ldots, F^{(N)}.$

	\For{ $k$ from $1$ to $n$ }
		\State{ $\left(x_1, \cdots, x_n\right) \gets $ sample from measure $L_n^{p_1,p_2}$}	
		\State{ $F^{(k)} \gets n A_n \prod_{i=1}^{n-1} \left(x_n - x_i \right)^{\beta} x_n^{r_{1,n}-1} \left(1-x_n \right)^{r_{2,n}-1} h^{-1}\left(x_n \right)$}
	\EndFor
	\State{ Est.  $ \gets \sum_{ k=1}^nF^{(k)}/N $ }
	\State{ Std.  $ \gets \sum_{ k=1}^n\left(F^{(k)} - Est. \right)^2/N $ }
	\State{ C.O.V.  $ \gets$ Est. / Std.   }
	\State{\textbf{return} (Est., C.O.V. ) }
\end{algorithmic}
\end{breakablealgorithm}
\vspace{1em}

Next, we computed the probabilities $\mathbb{P}\left(\lambda_{(n)}> \frac{p_1x}{p}\right)$ with three sets of parameters $\left( \beta, n, p_1, p_2\right)$:
\begin{align*}
\left( \beta, n, p_1, p_2\right) = & \left(1, 10, 10^3, 10^9\right), \left(10, 10, 10^3, 10^9\right), \left(100, 10, 10^3, 10^9\right).
\end{align*}
These parameter configurations correspond to ${\bf H}_0$, ${\bf H}_1$, and ${\bf H}_2$ conditions, respectively, as stated in Theorem \ref{thmsimu}.
%Under ${\bf H}_1$ and ${\bf H}_2$, the estimators are strongly efficient and asymptotically efficient. Particularly, when ${\bf H}_0$ is considered, the estimator is not only strongly efficient but also converges to the optimal sampling measure in total variation.
For convenience, we refer to these parameter settings as ${\bf H}_0$, ${\bf H}_1$, and ${\bf H}_2$ throughout the rest of this section.

Based on $10^4$ independent simulations and using the previously mentioned algorithm along with the R language, we carried out a series of simulation experiments. The left side of each graph shows the relationship between the coefficient of variation of sampled data $F^{(1)}, \ldots, F^{(N)}$ and $x$. The right side of each graph depicts the relationship between the estimated values of $\mathbb{P}\left(\lambda_{(n)}> x \right)$ and $x$ on a logarithmic scale axis.

R code of the proposed algorithm can be found at \href{}{https://github.com/S1yuW/Rare-event.}

\begin{figure}[H]\label{log1}
	\includegraphics[width=0.9\textwidth]{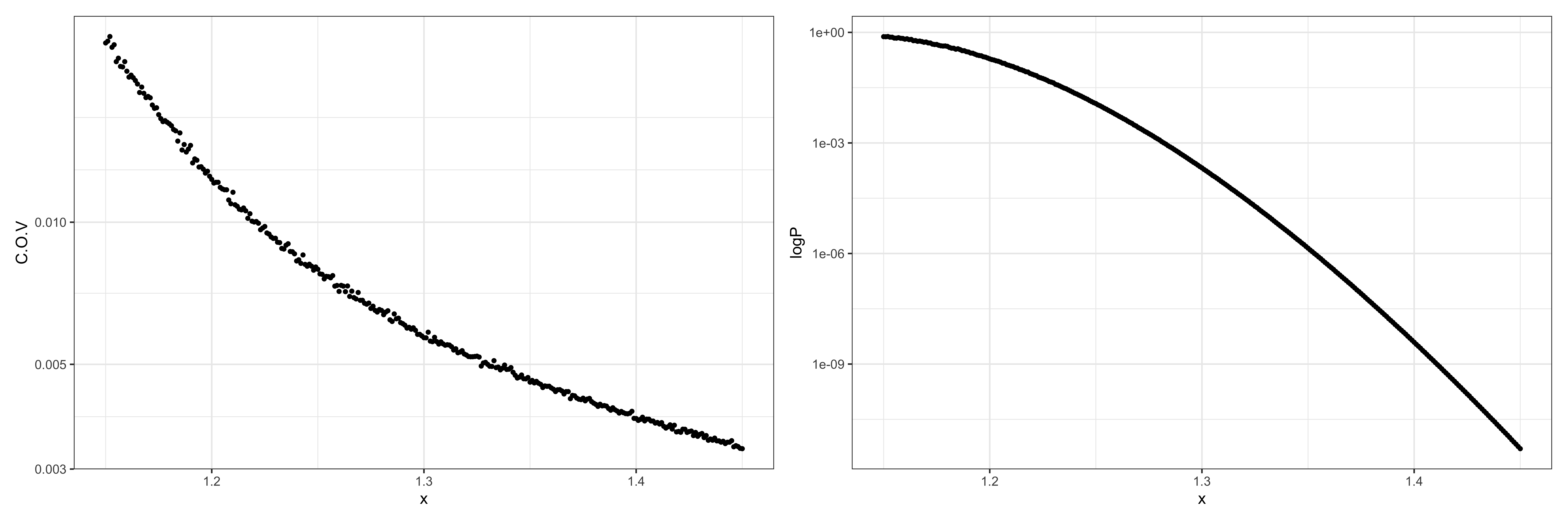}
	\caption{${\bf H}_0$: $\beta =1$, $n=10$, $p_1=10^3$ and $p_2=10^9$.}
\end{figure}

\begin{figure}[H]\label{log2}
	\includegraphics[width=0.9\textwidth]{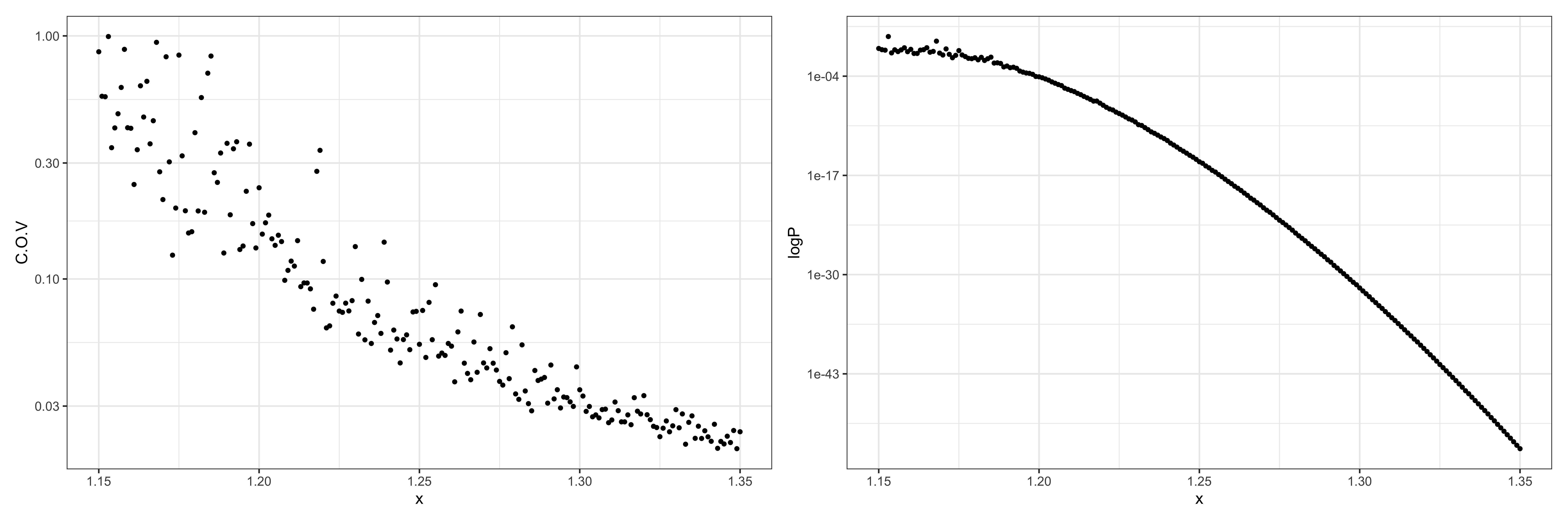}
	\caption{${\bf H}_1$: $\beta =10$, $n=10$, $p_1=10^3$ and $p_2=10^9$.}
\end{figure}

\begin{figure}[H]\label{log3}
	\includegraphics[width=0.9\textwidth]{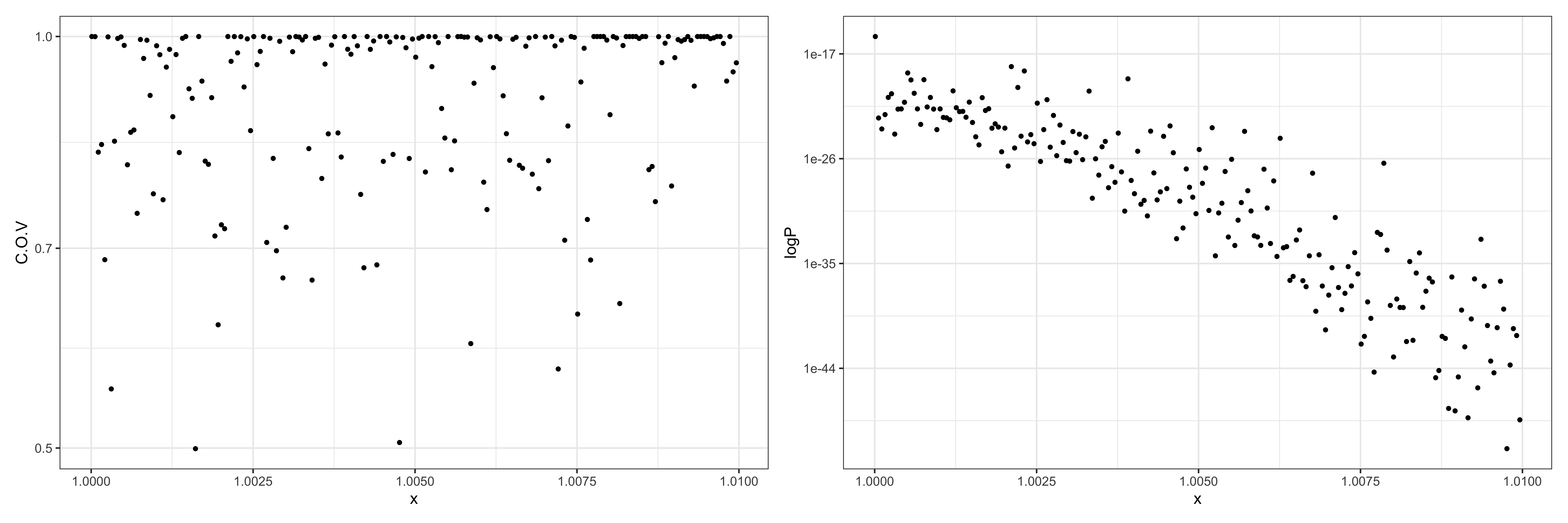}
	\caption{${\bf H}_2$: $\beta =100$, $n=10$, $p_1=10^3$ and $p_2=10^9$.}
\end{figure}

Under ${\bf H}_0$ and ${\bf H}_1$, the decreasing coefficient of variation with increasing $x$ means the stability and precision of the importance sample (IS) estimator. In the context of ${\bf H}_0$, the relatively smooth curve of the estimated values reflects the stability and consistency of the IS estimator. Specifically, when rare events are not exceedingly rare under ${\bf H}_1$, the coefficient of variation nearly approaches 1, leading to noticeable fluctuations in the IS estimator's values curve. By contrast, under ${\bf H}_0$, the coefficient of variation remains relatively small, maintaining the smoothness of the curve. This emphasizes that the IS estimator under ${\bf H}_1$ is more sensitive and exhibits stronger uncertainty compared to the IS estimator under ${\bf H}_0$. Furthermore, the IS estimator under ${\bf H}_0$ demonstrates stronger consistency compared to the IS estimator under ${\bf H}_1$. Finally, under ${\bf H}_2$, the estimated values curve displays a general declining trend, but the overall instability in the IS estimator's values emphasizes the differences between a weakly efficient IS estimator and a strongly efficient one.

\section{Technical lemmas}\label{Stechlem}
Before the proof of Theorem \ref{thmsimu}, some lemmas needed are listed in Section \ref{Stechlem}. We will prove Theorem  \ref{thmsimu} in Sections \ref{Smainthm}. And the proofs of lemmas in the previous sections will be presented in the supplementary material.

We first present a basic inequality, borrowed from \cite{JiangRare}.
\begin{lem}\label{loglem}
For $\alpha_1>1 / 2$ and $0<\alpha_2<1 / 2$, we have
\begin{align*}
	&\log (1-z) \geq-z-\alpha_1 z^2, \quad \text { if } z \in\left(-1,1-\frac{1}{2 \alpha_1}\right); \\
	&\log (1+z) \leq z-\alpha_2 z^2, \quad \text { if } z \in\left(-1, \frac{1}{2 \alpha_2}-1\right). 
\end{align*}  
\end{lem}

%%%%%%%%%%%%%%%%%%%%%%%%%%%%%%%%%%%%%%%%%%%%%%%%%%%%
%%%%%%%%%%%%%%%%%%%%%%%%%%%%%%%%%%%%%%%%%%%%%%%%%%%%

The following lemma provides some basic properties of the beta distributions. The first two results are clear and the third one is taken from \cite{Nagel}.

\begin{lem}\label{betalem}
Let $X$ be a Beta-distributed random variable with parameter $a$ and $b,$ denoted by $X\sim {\rm Beta}(a, b),$ that equivalently means the density function of $X$ is $$\frac{\Gamma(a+b)}{\Gamma(a)\Gamma(b)}x^{a-1} (1-x)^{b-1}1_{0<x<1}.$$ The following three assertions hold true.  
\begin{itemize}
	\item [(i)] The $k$-th moment of $X$ is
\begin{align*}
	\mathbb E \left[X^k\right]  = \prod_{i=1}^{k} \frac{ a +i-1}{a+b+i-1}.
\end{align*}  
	\item [(ii)] The variance of $X$ and $X^2$ are
\begin{align*}
\operatorname{Var}[X] & = \frac{a b}{(a+b)^2(a+b+1)}, \\
\operatorname{Var}\left[X^2\right] & =\frac{a(a+1)\left(4 a b(a+b+3)+6 b^2+b\right)}{(a+b)^2(a+b+1)^2(a+b+2)(a+b+3)}.
\end{align*}  
	\item[(iii)] For any $0 < \varepsilon < 1/2,$
$$
\mathbb P(|X- \mathbb E[X]|>\varepsilon) \leq 4 \exp \left\{-\frac{\varepsilon^2}{128} \cdot \frac{a^3+b^3}{a b}\right\}.
$$
\end{itemize}
\end{lem}

\vspace{1em}

Now, we are going to state some lemmas on the asymptotic of  $\mathbb{P}\left(\lambda_{(n)}>\frac{p_1}{p} x\right).$ To shorten the formula, set 
\begin{equation}\label{defun} u_n(x)=x^{r_{1, n}-1}\left(1- x\right)^{r_{2, n}-1}\end{equation}
with $r_{i, n}=\beta(p_i-n+1)/2$ for $i=1, 2.$
Throughout our analysis, we frequently use the factorization of the joint probability density function in \eqref{od} as below:
\begin{align}\label{orderden}
\mathrm{d} G_n^{p_1, p_2} =n A_n^{p_1, p_2} \prod_{i=1}^{n-1}\left(x_n-x_i\right)^\beta u_n(x_n) 1_{\left\{x_n>x_{n-1}\right\}} \mathrm{d} x_n \mathrm{~d} G_{n-1}^{p_1-1, p_2-1},
\end{align}  
 where $A_n^{p_1, p_2}$ is defined by
 \begin{align}\label{defAn} 
A_n^{p_1, p_2} = \frac{\Gamma\left(1+\frac{\beta}{2}\right) \Gamma\left(\frac{\beta p}{2}\right) \Gamma\left(\frac{\beta(p-1)}{2}\right)}{\Gamma\left(1+\frac{\beta n}{2}\right) \Gamma\left(\frac{\beta p_1}{2}\right) \Gamma\left(\frac{\beta p_2}{2}\right) \Gamma\left(\frac{\beta(p-n)}{2}\right)}.
 \end{align}
Here, $G_{n-1}^{p_1-1, p_2-1}$ represents the joint distribution of the order statistics of $\beta$-Jacobi ensemble $\mathcal{J}_{n-1}\left(p_1-1, p_2-1\right)$. 

Observing the terms in \eqref{orderden} one by one, we give first the asymptotical expression of $\log A_n^{p_1, p_2}$ under the condition {\bf A}, where {\bf A} is given as 
$$
{\bf A}: \quad \quad \quad  \lim_{n\to \infty} \frac{n}{p_1} = 0, \quad  \lim_{n\to \infty} \frac{p_1}{p_2} = \sigma \in [0, \infty) \quad \text{ and } \quad  \beta n \!>\!>1.
$$
 
\begin{lem}\label{Ann}
Let  $A_n^{p_1, p_2}$ be defined by \eqref{defAn}. 
Under the assumption {\bf A}, we have 
$$
	\log A_n^{p_1, p_2} =  \frac{\beta p_1}{2} \log \frac{p}{p_1} +\frac{\beta p_2}{2} \log \frac{p}{p_2}  +\frac{\beta (n-1)}{2} \log \frac{p}{n}  + \frac{1}{2}\log (\beta p_1) + \frac{1}{2}\beta n+o(\beta n). 
$$
\end{lem} 

\begin{proof}
By definition,
\begin{align*} 
	\log A_{n}^{p_1, p_2}
	=& \log \Gamma\left(1+\frac{\beta}{2}\right)+\log \Gamma\left(\frac{\beta p}{2}\right)+\log \Gamma\left(\frac{\beta(p-1)}{2}\right) \\
	& -\log \Gamma\left(1+\frac{\beta n}{2}\right)-\log \Gamma\left(\frac{\beta p_1}{2}\right)-\log \Gamma\left(\frac{\beta p_2}{2}\right)-\log \Gamma\left(\frac{\beta(p-n)}{2}\right)  .
\end{align*}
The Stirling formula says that 
$$
\log \Gamma(x)=\left(x-\frac{1}{2}\right) \log x-x+ \frac{1}{2} \log\left(2 \pi\right)+o(1)
$$ 
for sufficiently large $x$. Then, by combining alike terms, we see that
\begin{align}
	\log A_{n}^{p_1, p_2}=
	&\log \Gamma\left(1+\frac{\beta}{2}\right)-\frac{\beta}{2} \log \frac{\beta}{2}+  \frac{\beta p_1}{2} \log \frac{p}{p_1} +\frac{\beta p_2}{2} \log \frac{p}{p_2} \notag \\ 
	& +  \left(\frac{\beta (p-n)}{2}-\frac{1}{2}\right) \log \frac{p-1}{p-n} +\frac{\beta(n-1)}{2} \log \frac{\beta(p-1)}{2+\beta n} \notag \\
	&+ \frac 12 \log \frac{\beta p_1}{2} - \frac{\beta+1}{2}\log \frac{n}{2}- \frac{1}{2} \log\left(2 \pi\right) +o(1) . \label{logAno1}
\end{align}
Since $\log \Gamma\left(1+\frac{\beta}{2}\right)-\frac{\beta}{2} \log \frac{\beta}{2}$ is either $O(\beta)$ for $\beta$ large enough or bounded for finite $\beta$ and $$  \left(\frac{\beta (p-n)}{2}-\frac{1}{2}\right) \log \frac{p-1}{p-n} = \frac{1}{2}\beta n+o(\beta n) ,$$
it comes 
\begin{align*} 
	\log A_{n}^{p_1, p_2} =  \frac{\beta p_1}{2} \log \frac{p}{p_1} +\frac{\beta p_2}{2} \log \frac{p}{p_2}  +\frac{\beta (n-1)}{2} \log \frac{p}{n} + \frac{1}{2}\log \beta p_1 + \frac{1}{2}\beta n+o(\beta n).
\end{align*}
\end{proof}

Now, we state a lemma on integrals of some strictly convex  function. 
\begin{lem}\label{est1}
	Suppose $0 \le m< a<M \leq1$ and $k \geq 0$ and $b, c>1.$ Let $$\omega(t)= -k\,t -(b-1) \ln t-(c-1) \ln (1-t), \quad 0<t<1.$$
\begin{itemize}
\item[(i)] If $ \omega^{\prime}(a) >0,$ then we have 
	\begin{align*}
		\int_{a}^M e^{-\omega(t)} \mathrm{~d} t \leq \frac{e^{-k \,a} a^{b-1}(1-a)^{c-1} }{ \omega^{\prime}(a) } 	\end{align*}
and 
\begin{align*} 
	\int_a^M e^{-\omega(t)} \mathrm{~d} t  
	\geq & \frac{ e^{-k \,a} a^{b-1}(1-a)^{c-1}}{ \omega^{\prime}(M) }
			\left(1-  \exp \left\{ -(M-a)  \omega^{\prime}(a) \right\} \right). 
\end{align*}
\item[(ii)] By contrast,  it holds that 
	\begin{align*}
		\int_{m}^a e^{-\omega(t)} \mathrm{~d} t \leq  \frac{e^{-k \,a} a^{b-1}(1-a)^{c-1}}{ - \omega^{\prime}(a) }
	\end{align*}
and 
	\begin{align*}
		\int_{m}^a e^{-\omega(t)} \mathrm{~d} t \geq  \frac{e^{-k \,a} a^{b-1}(1-a)^{c-1}}{ - \omega^{\prime}(m) }\left(1-  \exp \left\{ (a-m) \omega^{\prime}(a) \right\} \right)
	\end{align*} 
	for $a$ when  $ \omega^{\prime}(a) <0.$
	\end{itemize}
\end{lem}
Next, based on Lemma \ref{est1}, we provide accurate estimates on integrals of $u_n(x)=x^{r_{1, n}-1}(1-x)^{r_{2, n}-1}.$

\begin{lem}\label{estintpxd}  
Suppose that $n<\!<\!p_1<\!<\!p_2.$ 
\begin{itemize}
\item[(i)] Given $x>1.$ For $z_n>0$ with $\frac n{p_1}<\!<\! z_n<\!<\! 1$ and $\kappa_n=o(\beta p),$ we have 
\begin{align*}
		 \int_{ \frac{p_1 x}{p} }^{\frac{p_1(x+z_n)}{p}} \exp\left\{\kappa_n \left(y-\frac{p_1 x}{p}\right)\right\} u_n\left(y\right) \mathrm{~d}y=
		\frac{2 x}{(x-1) \beta p }u_n\left(\frac{p_1 x}{p}\right)(1+o(1)).\end{align*}   
\item[(ii)] Given $z_n>0$ with  $\frac{n}{p_1}<\!<z_n\le M<\infty$  and let $l$ be a positive integer. Supposing that $d_n=O(\beta n)$ and  the sequence $\kappa_n$ satisfying $$\lim_{n\to\infty}\left(-\frac{\kappa_n}{\beta p z_n}+\frac{l}{2(1+z_n)}\right)=c>0,$$ we claim that 
\begin{align*}
	 &\quad \int_{\frac{p_1}{p}\left(1+z_n \right)}^{1}   \exp \left\{  \kappa_n\left(y-\frac{p_1(1+z_n)}{p}\right) \right\}  u_n^l(y) y^{d_n}\mathrm{~d} y \\
	  &	 \leq \frac{1}{c\, \beta p z_n} u_n^l\left(\frac{p_1(1+z_n)}{p}\right)\left(\frac{p_1(1+z_n)}{p}\right)^{d_n}(1+o(1)).
\end{align*}  
\item[(iii)] Given $z_n>0$ with $\frac{n}{p_1}<\!<z_n<\!<1$ and $\kappa_n=o(\beta p),$ we have 
\begin{align*}
	& \int_{ 0}^{ \frac{p_1(1-z_n)}{p} }  \exp \left\{ \kappa_n\left(y-\frac{p_1(1-z_n)}{p}\right) \right\}  u_n(y) \mathrm{~d} y
	\leq \frac{2}{\beta p z_n} u_n\left(\frac{p_1(1-z_n)}{p}\right)(1+o(1)). 
\end{align*}  
\end{itemize} 
\end{lem}

In order to treat the terms in the expression \eqref{orderden}, it remains to investigate the interaction $\prod_{i=1}^{n-1}|x_n-x_i|,$ which produces the main difficulties.  For that aim, we introduce a particular set $E_n,$ defined as 
\begin{align}
	E_n = \left\{(x_1, \cdots, x_n) \in \Delta_n \left| \frac{p_1}{p}(1-\delta_n ) < x_1, x_{n-1} <  \frac{p_1}{p}(1+\delta_n ) ,  \frac{p_1}{p} x < x_n < \frac{p_1}{p} \left(x +\delta_n \right) \right. \right\}. \label{DefEn} 	
\end{align}
Here, $\Delta_n:=\{(x_1, \cdots, x_n)\in [0, 1]^n: x_1\le x_2\le \cdots\le x_n \}.$
\begin{lem}\label{interaction} Given $x>1$ and $0<\delta_n<\!<\! 1$ and let $E_n$ be defined as above. Set $$\alpha_2 = \frac{1}{2} - \frac{\delta_n}{x-1}, \quad \alpha_1 = \frac{1}{2}+ \frac{\delta_n}{x-1} \quad \text{and} \quad \bar{x}_i=\frac{px_i-p_1}{p_1(x-1)}.$$ For all $(x_1, \ldots, x_n) \in E_n,$ the interaction $\prod_{i=1}^{n-1}\left(x_n-x_i\right)$ has lower and upper bounds as follows
\begin{align*}
	\prod_{i=1}^{n-1}\left(x_n-x_i\right) &\leq   \left(\frac{p_1(x-1)}{p}\right)^{n-1} \exp \left\{  (n+o(n))\frac{p x_n -p_1x}{p_1(x-1) } - \sum_{i=1}^{n-1}\bar{x}_i - \alpha_2 \sum_{i=1}^{n-1}\bar{x}_i^2  \right\};\\
	\prod_{i=1}^{n-1}\left(x_n-x_i\right) &\geq
	\left(\frac{p_1(x-1)}{p}\right)^{n-1} \exp  \left\{-\sum_{i=1}^{n-1} \bar{x}_i-\alpha_1 \sum_{i=1}^{n-1}\bar{x}_i^2\right\}.
\end{align*}  

	\end{lem}

Lemma \ref{interaction} guides us to work on the following two lemmas, whose proofs are relatively long and are postponed to the last section. 

\begin{lem}\label{mainlem}
	Let $\left(\lambda_1, \lambda_2, \cdots, \lambda_n\right)$ be the $\beta$-Jacobi ensemble $\mathcal J_n(p_1, p_2)$ with density function \eqref{defjacobi}. 
For $n$ large enough, we have 
	\begin{align*}
		&\mathbb E\left[ \sum_{i=1}^n\left( \frac{p\lambda_i-p_1}{p_1} \right) \right]=O\left(\frac{n^2}{p}\right);\\	
		&	
		\mathbb E \left[\sum_{i=1}^n\left( \frac{p\lambda_i-p_1}{p_1} \right)^2\right] = \frac{n^2}{p_1}+O\left(\frac{n^2}{p}+\frac{n}{\beta p_1}+\frac{n^2}{p_1^2}\right).
	\end{align*}
\end{lem}

\begin{lem}\label{lemexp1}
Let $\left(\lambda_1, \lambda_2, \cdots, \lambda_n\right)$ be the $\beta$-Jacobi ensemble $\mathcal J_n(p_1, p_2)$ with density function \eqref{defjacobi} and suppose that $\beta np_1^2\ll p_2, \; n\ll p_1$ and $(\beta n)^3\ll (\beta p_1)^2.$  
We have the following three statements. 
\begin{enumerate} 
\item Given a positive sequence $\delta_n$ with $\sqrt{\frac{n}{p_1}}\ll \delta_n\ll 1$ and $\alpha=O\left( \frac{\beta}{\delta_n} \right).$ It holds that
\begin{align}
	\lim _{n \rightarrow \infty} \mathbb{E}\exp \left\{-\alpha \sum_{i=1}^n\left(\frac{p \lambda_i-p_1}{p_1}\right)\right\}= 1. \label{alphalimit}
\end{align}
\item For the case when $\eta_1=O(\beta)$ and $\eta_2=O(\beta),$ we have that 
\begin{align}
	\lim _{n \rightarrow \infty} \mathbb{E}\left[\exp \left\{-\eta_1 \sum_{i=1}^n\left(\frac{p \lambda_i-p_1}{p_1}\right)-\eta_2 \sum_{i=1}^n\left(\frac{p \lambda_i-p_1}{p_1}\right)^2+\frac{\eta_2 n^2}{p_1}\right\}\right] = 1. \label{etalimit}
\end{align}  
\item Given a positive sequence $\delta_n$ satisfying 
	$\sqrt{n/p_1}\ll \delta_n\ll 1$ and $0<\eta_1=O(\beta), 0<\eta_2=O(\beta).$ It holds that 
\begin{align}
	& \mathbb E\left[ \exp\left\{ -\eta_1\sum_{i=1}^n\left( \frac{p\lambda_i-p_1}{p_1} \right) - \eta_2 \sum_{i=1}^n\left( \frac{p\lambda_i-p_1}{p_1} \right)^2  \right\}  \mathbf 1_{ \left\{ \lambda_{(1)}> \frac{p_1}{p}(1-\delta_n ),  \lambda_{(n)} < \frac{p_1}{p}(1 + \delta_n ) \right\} } \right] \notag\\
	\ge  & \exp \left\{-\frac{\eta_2 n^2}{p_1}+o\left(1\right) \right\} \label{lowerforE}
\end{align}
for $n$ large enough. 
\end{enumerate}
\end{lem}

\section{Proof of Theorem \ref{thmsimu}}  \label{Smainthm}

To ease notations, we use $\mathbb E$ and $\mathbb E^L$ to refer to the expectations with respect to $G_{n}^{p_1, p_2}$ and $L_n^{p_1, p_2},$  and use $G$ and $L$ for corresponding measures, respectively. Recall that the IS estimator 
\begin{align*}
	F_n  =  n A_n^{p_1, p_2} \prod_{i=1}^{n-1} \left( \lambda_{(n)}- \lambda_{(i)} \right)^{\beta}
	u_n\left(  \lambda_{(n)} \right)h^{-1}(\lambda_{(n)}) \mathbf{1}_{\left\{ \lambda_{(n)} > p_1 x/p  \right\}}  ,
\end{align*}
where the function $h$ is given by 
$$
	h\left(y\right):= \frac{r}{1- e^{ -r(1-p_1x/p) }}  e^{-r\left(y-\left( p_1x /p \right) \vee \lambda_{(n-1)}  \right)} \mathbf 1_{\{  \left( p_1x /p \right) \vee \lambda_{(n-1)} < y< 1 \}}
$$with rate $r  = \frac{\beta p(x-1)}{2x}.$ 
We first prove the result when {\bf H}$_0$ holds. 
\subsection{Proof of the third assertions of Theorem \ref{thmsimu} when {\bf H}$_0$ is true.}
For any $ A \in \mathcal B^n,$ we know
\begin{align*}
	 \left|L_n^{p_1, p_2}(A)-\mathbb{P}\left(A \left| \lambda_{(n)}>\frac{p_1}{p} x \right.\right)\right|  
	=&   \left| \int_A 1 - \frac{F_n}{\mathbb{P}\left(\lambda_{(n)}>\frac{p_1}{p} x \right) } \mathrm{ d} L_n^{p_1, p_2}\right| \\
	\leq & \left\{\mathbb E^L \left[ \left( 1 - \frac{F_n}{\mathbb{P}\left(\lambda_{(n)}>\frac{p_1}{p} x \right) } \right)^2 \right]\right\}^{1/2} \\
	\leq &  \left\{ \frac{\mathbb E^L \left[  F_n^2 \right]}{\mathbb{P}\left(\lambda_{(n)}>\frac{p_1}{p} x \right)^2 } -1 \right\}^{1/2}.
\end{align*}  
As long as   
\begin{align}
\lim_{n \to \infty} \frac{\mathbb E^L\left[ F_n^2 \right]}{\mathbb P^2\left(\lambda_{(n)} > \frac{p_1}{p}x \right)} = 1, \label{limequa1}
\end{align}
it follows that 
$$
\lim _{n \rightarrow \infty} \sup _{A \in \mathcal{F}}\left|L_n^{p_1, p_2}(A)-\mathbb{P}\left(A \left| \lambda_{(n)}>\frac{p_1}{p} x \right.\right)\right| = 0.
$$
Thus, for the third assertions of Theorem \ref{thmsimu}, it remains to prove \eqref{limequa1}. 

Jensen's inequality guarantees the inequality 
\begin{align*}
	\mathbb{E}^L\left[F_n^2\right] = \mathbb{E}^L\left[F_n^2;  \lambda_{(n)}>\frac{p_1}{p}x \right]  \geq\left( \mathbb{E}^L\left[F_n;  \lambda_{(n)}>\frac{p_1}{p}x \right] \right)^2 = \left(\mathbb{P}\left(\lambda_{(n)}>\frac{p_1}{p}x\right)\right)^2.
\end{align*}  
 Recall the condition \begin{align*}
{\bf H}_0:	\beta n p_1^2\ll p_2, \quad \frac{\log n}{n}\vee \frac{n}{p_1}\ll \beta  \quad \text{and}  \quad (\beta n)^5\ll (\beta p_1)^3.
\end{align*}
Given a positive sequence $\delta_n$ satisfying $\sqrt{\frac{n}{p_1}}\ll \delta_n \ll \frac{p_1}{\beta n^2}\wedge 1,$ whose existence is allowed by $(\beta n)^5\ll (\beta p_1)^3$  and remember   
\begin{align*}
	E_n = \left\{ (x_1, \cdots, x_n) \in \Delta_n \left| \frac{p_1}{p}(1-\delta_n ) < x_1, x_{n-1} <  \frac{p_1}{p}(1+\delta_n ) ,  \frac{p_1}{p} x < x_n < \frac{p_1}{p} \left(x +\delta_n \right) \right. \right\},
\end{align*} 
which is obvious a subset of $\{x_n> \frac{p_1 x}{p}\}.$ Since $\mathbb{E}^{L}[F_n^2]=\mathbb{E}[F_n]$, for \eqref{limequa1}, it suffices to prove the inverse direction 
\begin{align}\label{ELFN2}
	\varlimsup_{n \rightarrow \infty} \frac{\mathbb{E}\left[F_n\right]}{\mathbb{P}^2\left(E_n\right)}\leq1.
\end{align}  
We first work on the probability $\mathbb{P}(E_n).$  The decomposition  \eqref{orderden} and Lemma \ref{interaction} bring that 
\begin{align*}
	\mathbb P(E_n)=&n A_n^{p_1, p_2}\int_{E_n}  \prod_{i=1}^{n-1}\left(x_n-x_i\right)^\beta u_n(x_n) \mathrm{~d}x_{n} \mathrm{~d} G_{n-1}^{p_1-1, p_2-1}\\
	 \geq
	&n A_n^{p_1, p_2} \left(\frac{p_1 (x-1)}{p}\right)^{\beta(n-1)} \int_{ p_1 x/p }^{p_1(x+\delta_n)/p}  u_n(x_n)\mathrm{~d}x_{n}       \\
	& \times \int_{\Omega_{\delta_n}} \exp\left\{ - \frac{\beta}{x-1}  \sum_{i=1}^{n-1}\frac{px_i-p_1}{p_1} -  \frac{\alpha_1 \beta}{(x-1)^2} \sum_{i=1}^{n-1}\left(\frac{px_i-p_1}{p_1}\right)^2   \right\} \mathrm{~d} G_{n-1}^{p_1-1, p_2-1} ,
\end{align*}  
where $$\Omega_{\delta_n} :=  \left[\frac{p_1(1-\delta_n )}{p}, \; \frac{p_1(1+\delta_n )}{p} \right]^{n-1}, \quad \alpha_1=\frac{1}{2}+\frac{\delta_n}{x-1} . $$
Lemmas \ref{estintpxd} and \eqref{lowerforE}  enable us to obtain 
\begin{align}
	\notag \mathbb P(E_n) 
	\geq & \frac{2 n A_n^{p_1, p_2} }{\beta p}\left(\frac{p_1(x-1)}{p}\right)^{\beta(n-1)}u_n\left( \frac{p_1 x}{p} \right)\times \exp\left\{-\frac{\alpha_1 \beta n^2}{p_1(x-1)^2}+O\left(\frac{\beta n^3}{p_1^2}\vee\frac{\beta n^2}{p}\right) \right\} \notag \\
	=&:R_n^{p_1, p_2}(x)\times \exp\left\{-\frac{\beta n^2}{2 p_1(x-1)^2}+o\left(1\right) \right\}. \label{lowforE}
\end{align}  
 Here, we use the condition $\beta \delta_n n^2=o(p_1)$ to write  $$\frac{\alpha_1 \beta n^2}{p_1(x-1)^2}=\frac{\beta n^2}{2p_1(x-1)^2}+\frac{\beta \delta_n n^2}{p_1(x-1)^3}=\frac{\beta n^2}{2p_1(x-1)^2}+o(1).$$

Next, we are going to reveal the veil of $\mathbb{E} [F_n].$ 
Note that 
	$\Delta_m$ is the disjoint union of 
	$$\{x_{m}\ge p_1x/p\}, \; \{p_1(1+\delta_n)/p\le x_{m}<p_1x/p\}, \;\{x_m<p_1(1+\delta_n)/p, x_1\ge p_1(1-\delta_n)/p\} $$ and $\{x_m<p_1(1+\delta_n)/p, x_1<p_1(1-\delta_n)/p\}.$
	
	By the partition theorem and \eqref{lowforE}, for \eqref{ELFN2}, it suffices to prove the following four inequalities
\begin{align}
&\mathbb{E}[F_n; \lambda_{(1)}>\frac{p_1}{p}(1-\delta_n), \lambda_{(n-1)} < \frac{p_1}{p}(1+\delta_n) ]	\le (R_n^{p_1, p_2}(x))^2\times \exp\left\{-\frac{\beta n^2}{p_1(x-1)^2}+o\left(1\right) \right\}, \label{upperfor1} \\
	& \mathbb{E}\left[ F_n ;  \lambda_{(n-1)}>p_1 x/p  \right] = o(1) (R_n^{p_1, p_2}(x))^2\times \exp\left\{-\frac{\beta n^2}{p_1(x-1)^2}\right\}, \label{fn1}\\
	 	 &\mathbb{E}\left[ F_n ; p_1 x/p >  \lambda_{(n-1)}> p_1(1+\delta_n)/p \right] = o(1) (R_n^{p_1, p_2}(x))^2\times \exp\left\{-\frac{\beta n^2}{p_1(x-1)^2}\right\}, \label{fn2} \\
		& \mathbb{E}\left[ F_n ;  \lambda_{(1)}<p_1(1-\delta_n)/p, \lambda_{(n-1)}<p_1 (1+\delta_n)/p \right] = o(1) (R_n^{p_1, p_2}(x))^2\times \exp\left\{-\frac{\beta n^2}{p_1(x-1)^2}\right\}. \label{fn3}
\end{align}  
 
First, it holds obviously that   
\begin{align}
	 \mathbb{E}\left[F_n 1_{A}\right] &= \int_{A,\, x_n>p_1 x/p} n A_n^{p_1, p_2} \prod_{i=1}^{n-1}\left( x_n - x_i \right)^{\beta} u_n(x_n) h^{-1}(x_n)  \mathrm{~ d} G_{n}^{p_1, p_2} \notag \\
	&\leq  n^2 (A_n^{p_1, p_2})^2 r^{-1}\int_{A, \, x_n>p_1 x/p}  \prod_{i=1}^{n-1}\left( x_n - x_i \right)^{2\beta} u_n^2(x_n)  e^{r\left( x_n- p_1x/p \right)} \mathrm{~d} x_n \mathrm{~d} G_{n-1}^{p_1-1, p_2-1} \label{generalform}\end{align} 
	for any $A\in \mathcal{B}^n.$ We will prove these inequalities by the order of showing up based on \eqref{generalform}.
	
{\bf Proof of \eqref{upperfor1} }. Set
$$
\Omega_{1, F} := \left\{ (x_1, \cdots, x_n) \in \Delta_n\bigg|   x_1 > \frac{p_1}{p}(1-\delta_n), x_{n-1} < \frac{p_1}{p}(1+\delta_n), x_n > \frac{p_1}{p}x \right\} 
$$
and set $ \alpha_2 = \frac{1}{2} - \frac{\delta_n}{x-1}.$ The inequality \eqref{generalform} and Lemma \ref{interaction} help us to write that 
\begin{align}
	{\rm I}_{1, n}:=& \mathbb{E}\left[F_n ; \lambda_{(1)} > p_1(1-\delta_n )/p, \lambda_{(n-1)} < p_1(1 +\delta_n )/p \right] \notag\\
	\le & n^2 (A_n^{p_1, p_2})^2 r^{-1}\int_{\Omega_{1, F}}  \prod_{i=1}^{n-1}\left( x_n - x_i \right)^{2\beta} u_n^2(x_n)  e^{r\left( x_n- x_{n-1} \right)} \mathrm{~d} x_n \mathrm{~d} G_{n-1}^{p_1-1, p_2-1}\notag\\
	\leq & n^2 (A_n^{p_1, p_2})^2 r^{-1}\left(\frac{p_1 (x-1)}{p}\right)^{2\beta(n-1)} \int_{p_1x/p}^1 \exp \left\{ v_n\frac{p x_n -p_1x}{p_1x -p_1 }\right\} u_n^2(x_n)  e^{r\left( x_n- p_1x/p \right)} \mathrm{~d} x_n \notag \\
	& \times \int_{[0, 1]^{n-1}}  \exp \left\{- 2\beta \sum_{i=1}^{n-1}\frac{px_i-p_1}{p_1 (x-1)} - 2\beta \alpha_2 \sum_{i=1}^{n-1}\left(\frac{px_i-p_1}{p_1(x-1)}\right)^2  \right\}\mathrm{~d} G_{n-1}^{p_1-1, p_2-1}. \label{upforF}
\end{align}  
Here, $v_n=2\beta n(1+o(1)).$
 
 Taking $z_n=x-1, l=2$ and $\kappa_n=v_n /(p_1(x-1))+r,$ we see $$c=\lim_{n\to\infty}\left(-\frac{\kappa_n}{\beta p z_n}+\frac{l}{2(1+z_n)}\right)=\frac{1}{2x}.$$ Applying Lemma \ref{estintpxd}, we obtain that 
\begin{align*}
		\int_{p_1x/p}^1 \exp \left\{ v_n\frac{p x_n -p_1x}{p_1x -p_1 }\right\} u_n^2(x_n)  e^{r\left( x_n- p_1x/p \right)} \mathrm{~d} x_n \leq \frac{2x(1+o(1))}{\beta p(x-1)}u_n^2\left(\frac{p_1x}{p}\right).
\end{align*}  
On the other hand, a direct application of Lemma \ref{lemexp1} and the condition $\delta_n \beta n^2=o(p_1)$ entail that 
\begin{align*}
	& \int_{[0, 1]^{n-1}} \exp\left\{ - \frac{2 \beta}{x-1}  \sum_{i=1}^{n-1}\frac{px_i-p_1}{p_1} -  \frac{ 2  \beta \alpha_2}{(x-1)^2} \sum_{i=1}^{n-1}\left(\frac{px_i-p_1}{p_1}\right)^2  \right\} \mathrm{~d} G_{n-1}^{p_1-1, p_2-1} \\
	= &
	\exp\left\{ -\frac{2\alpha_2 \beta  n^2}{p_1(x-1)^2}+o(1) \right\}=\exp\left\{ -\frac{\beta  n^2}{p_1(x-1)^2}+o(1) \right\}.
\end{align*}  
Plugging all these bounds into \eqref{upforF} and comparing with the right hand side of \eqref{lowforE}, we have that 
\begin{align*}
	{\rm I}_{1, n}
	\leq  (R_{n}^{p_1, p_2}(x))^2 \exp\left\{ -\frac{\beta  n^2}{p_1(x-1)^2}+o\left(1\right) \right\}.
\end{align*}  
{\bf  Proof of \eqref{fn1}}.  
Set
$$
	\Omega_{2, F}:= \left\{ (x_1, \cdots, x_n) \in \Delta_n \bigg| x_n > \frac{p_1 x}{p}, x_{n-1} > \frac{p_1 x}{p}  \right\}.
$$  
Similar as the expression \eqref{generalform}, we see that   
$$\aligned 
{\rm I}_{2, n}:=& \mathbb{E}[F_n, \lambda_{(n-1)}>p_1 x/p]\\
 \le & n^2 (A_n^{p_1, p_2})^2 r^{-1}\int_{\Omega_{2, F}}  \prod_{i=1}^{n-1}\left( x_n - x_i \right)^{2\beta} u_n^2(x_n)  e^{r\left( x_n- x_{n-1} \right)} \mathrm{~d} x_n \mathrm{~d} G_{n-1}^{p_1-1, p_2-1}\\
	= & n^2 (A_n^{p_1, p_2})^2  (n-1)A_{n-1}^{p_1-1, p_2-1} r^{-1}\int_{\Omega_{2, F}} \prod_{i=1}^{n-1}   \left(x_n-x_i\right)^{2\beta} u_n^2(x_n) e^{r (x_n-x_{n-1})} \\
	&\times \prod_{j=1}^{n-2}\left(x_{n-1}-x_j\right)^\beta u_n(x_{n-1}) \times   \mathrm{~d} x_n \mathrm{~d} x_{n-1} \mathrm{~d} G_{n-2}^{p_1-2, p_2-2}, \label{into1}
\endaligned  $$
where in the second step we use again the decomposition as \ref{orderden}.
Use two simple inequalities $$  \prod_{i=1}^{n-1} \left(x_n-x_i\right) \leq x_n^{n-1}  \quad\text{and}\quad \prod_{j=1}^{n-2}\left(x_{n-1}-x_j\right) \leq x_{n-1}^{n-2}$$ for $(x_1, \cdots, x_n) \in \Omega_{2, F}$ to get
\begin{align*}
	{\rm I}_{2, n}&\leq n^2 (A_n^{p_1, p_2})^2  (n-1)A_{n-1}^{p_1-1, p_2-1} \int_{p_1 x/p}^1 x_n^{2 \beta(n-1)}u^2_n(x_n)e^{rx_n}\mathrm{~d} x_n\\
	&\quad \times 
	 \int_{p_1 x/p}^1 x_{n-1}^{\beta(n-2)}u_n(x_{n-1})e^{-rx_{n-1}}\mathrm{~d} x_{n-1}. 
\end{align*}  

Applying Lemma \ref{estintpxd} with $d_n=2\beta(n-1), \kappa_n=r$ and $z_n=x-1, l=2$, or $d_n=\beta(n-2), \kappa_n=-r, l=1$ respectively, we obtain that
\begin{align*}
	& \int_{p_1 x/p}^1 x_n^{2 \beta(n-1)}u^2_n(x_n)e^{rx_n}\mathrm{~d} x_n \leq \frac{2 x(1+o(1))}{\beta p(x-1)} e^{ r p_1 x/p} \left(\frac{p_1x}{p}\right)^{2 \beta(n-1)} u_n^2\left(\frac{p_1x}{p} \right);  \\
&	\int_{p_1 x/p}^1 x_{n-1}^{\beta(n-2)}u_n(x_{n-1})e^{-rx_{n-1}}\mathrm{~d} x_{n-1} 
	\leq \frac{x(1+o(1))}{\beta p(x-1)} e^{-r p_1x/p }\left(\frac{p_1x}{p}\right)^{\beta(n-2) } u_n\left(\frac{p_1x}{p} \right) .
\end{align*}  

Thus,
\begin{align*}
	I_{2, n}:=&\mathbb{E}\left[ F_n ;  \lambda_{(n-1)}>p_1 x/p  \right] \\
	\leq  & \frac{4 x^3 n^2 (A_n^{p_1, p_2})^2  (n-1)A_{n-1}^{p_1-1, p_2-1} }{\left(\beta p (x-1)\right)^3}\left(\frac{p_1x}{p}\right)^{\beta(3n-4)} u_n^3\left(\frac{p_1x}{p} \right) \\
	=&\frac{(n-1)A_{n-1}^{p_1-1, p_2-1}x^{3+\beta(3n-4)} }{\beta p(x-1)^{2\beta(n-1)+1}}(R_{n}^{p_1, p_2}(x))^2u_n\left(\frac{p_1x}{p} \right)\left(\frac{p_1}{p} \right)^{n-2}.
\end{align*}  
Thus, we have 
$$\aligned \log \frac{{\rm I}_{2, n}}{(R_{n}^{p_1, p_2}(x))^2}&\le \log A_{n-1}^{p_1-1, p_2-1}+(\beta(n-2)+r_{1, n}-1)\log\frac{p_1}{p}+(r_{2, n}-1)\log\left(1-\frac{p_1 x}{p}\right) \\
&\quad +\frac{\beta p_1}{2}\log x-\log(\beta p)+O(\beta n).
\endaligned $$
Furthermore, Lemma \ref{Ann} tells that 
\begin{align}\log A_{n-1}^{p_1-1, p_2-1} =  \frac{\beta \left(p_1-1\right)}{2} \log \frac{p}{p_1} +\frac{\beta p_2}{2} \log \frac{p}{p_2}  +\frac{\beta (n-2)}{2} \log \frac{p}{n}  + \frac{1}{2}\log (\beta p_1) + \frac{1}{2}\beta n+o(\beta n). \label{An-1}\end{align}
Putting the asymptotic of $\log A_{n-1}^{p_1-1, p_2-1}$ into the expression for $\log \frac{{\rm I}_{2, n}}{(R_{n}^{p_1, p_2}(x))^2}$ and combining alike terms, one gets that 
$$\aligned \log \frac{{\rm I}_{2, n}}{(R_{n}^{p_1, p_2}(x))^2}+\frac{\beta n^2}{p_1(x-1)^2}&\le \frac{\beta p_1}{2}(1-x+\log x)+o(\beta p_1).
\endaligned $$
Since $1-x+\log x<0$ for $x>1,$ it holds that $\log \frac{{\rm I}_{1, n}}{(R_{n}^{p_1, p_2}(x))^2}+\frac{\beta n^2}{(x-1)^2}\to -\infty,$ which is equivalent to the expression \eqref{fn1}. 

%%%%%%%%%%%%%%%%%%%%%%%%%%%%%%%%%%%%%
{\bf  Proof of \eqref{fn2}.}
Set
\begin{align*}
	\Omega_{3, F}:= \left\{ \mathbf (x_1, \cdots, x_n)\in \Delta_n \left| x_n > \frac{p_1 x}{p} > x_{n-1} >   \frac{p_1 (1+\delta_n )}{p} \right.  \right\}.
\end{align*}  
Similar as for ${\rm I}_{2, n},$ we have
\begin{align}
	{\rm I}_{3, n}:=&\mathbb{E}\left[ F_n ; p_1 x/p >  \lambda_{(n-1)}> p_1(1+\delta_n)/p \right]\notag \\
	\le & n^2 (A_n^{p_1, p_2})^2 r^{-1} (n-1)(A_{n-1}^{p_1-1, p_2-1})\int_{\Omega_{3, F}} \prod_{i=1}^{n-1}   \left(x_n-x_i\right)^{2\beta}  u_n^2(x_n) e^{r(x_n-p_1 x/p)} \notag\\
	&\times \prod_{j=1}^{n-2}\left(x_{n-1}-x_j\right)^\beta u_n(x_{n-1})     \mathrm{~d} x_n \mathrm{~d} x_{n-1} \mathrm{~d} G_{n-2}^{p_1-2, p_2-2}.     \label{I3form}        
\end{align} 
Using the inequality $\log(1+u) \leq u$ for $u >0$ , we have
\begin{align}\label{upxnxi}
	 \prod_{i=1}^{n-1}(x_n-x_i)  
	= & \left(\frac{p_1 x-p_1}{p}\right)^{n-1} \exp \left\{\sum_{i=1}^{n-1} \log \left(1+\frac{p x_n-p_1 x}{p_1 (x-1)}-\frac{p x_i-p_1}{p_1(x-1)}\right)\right\}  \notag \\
	\leq & \left(\frac{p_1(x-1)}{p} \right)^{n-1} \exp \left\{ \frac{(n-1) p}{p_1(x -1) }(x_n-p_1x/p) - \sum_{i=1}^{n-1}\frac{px_i-p_1}{p_1(x-1)}   \right\},
\end{align}  
and similarly, 
\begin{align}\label{upxn1xi}
	\prod_{j=1}^{n-2}(x_{n-1}-x_j) & \leq \left(\frac{p_1}{p}\delta_n\right)^{n-2} \exp \left\{ \frac{(n-2) p }{p_1 \delta_n }(x_{n-1}-p_1(1+\delta_n)/p) - \sum_{j=1}^{n-2}\frac{px_j - p_1}{p_1\delta_n} \right\} .
\end{align} 
Therefore, 
$$\aligned 
&\int_{\Omega_{3, F}} \prod_{i=1}^{n-1}   \left(x_n-x_i\right)^{2\beta}  u_n^2(x_n) e^{r(x_n-p_1 x/p)} \prod_{j=1}^{n-2}\left(x_{n-1}-x_j\right)^\beta u_n(x_{n-1}) \mathrm{~d} x_n \mathrm{~d} x_{n-1} \mathrm{~d} G_{n-2}^{p_1-2, p_2-2}\\
\le & \left(\frac{p_1(x-1)}{p} \right)^{2\beta(n-1)} \left(\frac{p_1}{p}\delta_n\right)^{\beta(n-2)} \int_{[0,1]^{n-2}}  \exp\left\{ -\left( \frac{2\beta}{x-1} +\frac{\beta}{\delta_n} \right) \sum_{i=1}^{n-2} \frac{px_i -p_1}{p_1}  \right\} \mathrm{~d} G_{n-2}^{p_1-2, p_2-2} \\
	& \times \int_{p_1x/p}^1 \exp\left\{ \left( \frac{2\beta p(n-1)}{p_1(x-1)} +r \right)\left(x_n -\frac{p_1}{p}x \right) \right\} u_n^2(x_n)\mathrm{~d} x_n \\
	& \times \int_{p_1(1+\delta_n)/p}^1 \exp\left\{ \left(\frac{\beta(n-2)p}{p_1 \delta_n} - \frac{2\beta p}{p_1(x-1)}\right)\left( x_{n-1}  - \frac{p_1}{p}\left(1+\delta_n \right) \right) \right\}  u_n(x_{n-1})     \mathrm{~d} x_{n-1}.\endaligned $$
 When $\kappa_n:=\frac{\beta(n-2)p}{p_1 \delta_n} -\frac{2\beta p}{p_1(x-1)},$ it follows from $\sqrt{n/p_1}<\!<\delta_n<\!<1$ and $p_1\delta_n\to\infty$ that 
 $$\lim_{n\to\infty}\left(-\frac{\kappa_n}{\beta p\delta_n}+\frac{l}{2(1+\delta_n)}\right)=\frac12.$$ Hence, Lemma \ref{estintpxd}   tells that 
 \begin{align*}
	 & \int_{p_1(1+\delta_n)/p}^1 \exp\left\{ \left(\frac{\beta(n-2)p}{p_1 \delta_n} - \frac{2\beta p}{p_1(x-1)}\right)\left( x_{n-1}  - \frac{p_1}{p}\left(1+\delta_n \right) \right) \right\}  u_n(x_{n-1})     \mathrm{~d} x_{n-1} \\
	 \leq & \frac{2(1+o(1))}{\beta p \delta_n} u_n\left( \frac{p_1}{p}\left(1+\delta_n \right) \right).
\end{align*}
Similarly, 
\begin{align}
	\int_{p_1 x/p}^1 \exp\left\{ \left( \frac{2\beta p(n-1)}{p_1(x-1)} +r \right)\left(x_n -\frac{p_1 x}{p} \right) \right\} u_n^2(x_n)\mathrm{~d} x_n 
	\leq  \frac{2x(1+o(1))}{\beta p(x-1)} u_n^2\left(\frac{p_1 x}{p}\right). \label{inteupto1}\end{align}  
Applying Lemma \ref{lemexp1} with $\alpha=\frac{2\beta}{x-1} +\frac{\beta}{\delta_n}$, we obtain that $$ \int_{[0,1]^{n-2}}  \exp\left\{ -\left( \frac{2\beta}{x-1} +\frac{\beta}{\delta_n} \right) \sum_{i=1}^{n-2} \frac{px_i -p_1}{p_1}  \right\} \mathrm{~d} G_{n-2}^{p_1-2, p_2-2}  = O(1).$$  
Plugging all these estimates into \eqref{I3form}, we have that 
$$\log\frac{{\rm I}_{3, n}}{(R_{n}^{p_1, p_2}(x))^2}\le \log\frac{A_{n-1}^{p_1-1, p_2-1}}{\beta p\delta_n}+\beta (n-2)\log\frac{p_1\delta_n}{p}+\log u_n\left(\frac{p_1(1+\delta_n)}{p}\right)+O(\beta n).$$ 
Examining the expression \eqref{An-1}, keeping in mind the condition $n/p_1<\!<\delta_n^2<\!<1,$ we have that  
\begin{align}
&\log\frac{{\rm I}_{3, n}}{(R_{n}^{p_1, p_2}(x))^2}\notag\\
\leq & \frac{\beta p_1}{2}\log \left(1+\delta_n \right) + \frac{\beta p_2}{2}\log \left(1-  \frac{p_1}{p_2}\delta_n \right) + \frac{\beta(n-2)}{2} \log \frac{p_1 \delta_n^2 }{n} - \frac{1}{2} \log \left( \beta p_1  \delta_n^2\right) +O\left( \beta n \right) \notag \\
\le &	- \frac{\beta p_1 \delta_n^2 }{4} + o\left( \beta p_1 \delta_n^2 \right) \label{logI3nn}
\end{align} 
and then the fact $(\beta n^2/p_1)\ll \beta p_1 \delta_n^2$ ensured by $n/p_1\ll \delta_n^2$ helps to obtain 
\begin{align*}
&\log\frac{{\rm I}_{3, n}}{(R_{n}^{p_1, p_2}(x))^2}+\frac{\beta n^2 }{p_1(x-1)}\le- \frac{\beta p_1 \delta_n^2 }{4} + o\left( \beta p_1 \delta_n^2 \right).
\end{align*}
Note that $\beta p_1 \delta_n^2 $ tends to infinity, which indicates that \eqref{fn2} hold.

%%%%%%%%%%%%%%%%%%%%%%%%%%%%%%%%%%%%%
{\bf  Proof of \eqref{fn3}.}
Set
\begin{align*}
	\Omega_{4, F}:= \left\{ (x_1, \cdots, x_n)\in \Delta_n \left| x_n > \frac{p_1 x}{p}> x_{n-1},  x_1 < \frac{p_1 (1 - \delta_n )}{p}  \right.  \right\}.
\end{align*}  
Similarly, using inequalities \eqref{upxnxi} and \eqref{upxn1xi} with $p_1(1+\delta_n)/p$ replaced by $p_1(1-\delta_n)/p$ for $(x_1, \cdots, x_n) \in \Omega_{3, F}$ to treat the interaction $\prod_{i=1}^{n-1} |x_n-x_i|,$ we have that 
\begin{align*}
	 {\rm I}_{4, n}:=&  \mathbb{E}\left[ F_n ;  \lambda_{(n-1)}<p_1 x/p, \lambda_{(1)}<p_1(1-\delta_n)/p \right]  \\
	 \leq & n^2 A_n^2  (n-1)A_{n-1} r^{-1}\int_{\Omega_{3, F}} \prod_{i=1}^{n-1}   \left(x_n-x_i\right)^{2\beta}  u_n^2(x_n) e^{r(x_n-p_1 x/p)} \prod_{j=2}^{n-1}\left(x_{j}-x_1\right)^\beta u_n(x_{1})  \notag \\
	& \quad \quad \quad \quad \quad \quad \quad \quad \quad \times   \mathrm{~d} x_n \mathrm{~d} x_{n-1} \mathrm{~d} G_{n-2}^{p_1-2, p_2-2}\\
	 \leq & n^2 A_n^2  (n-1)A_{n-1}r^{-1} \left(\frac{p_1x-p_1}{p} \right)^{2\beta(n-1)} \left(\frac{p_1}{p}\delta_n\right)^{\beta(n-2)} \exp\left\{ \frac{2\beta \delta_n}{x-1} \right\} \\
	& \times \int_{p_1x/p}^1 \exp\left\{ \left( \frac{2\beta p(n-1)}{p_1(x-1)} +r \right)\left(x_n -\frac{p_1}{p}x \right) \right\} u_n^2(x_n)\mathrm{~d} x_n \\
	& \times \int_0^{p_1(1-\delta_n)/p}\exp\left\{ - \left(\frac{\beta(n-2)p}{p_1 \delta_n} + \frac{2\beta p}{p_1(x-1)}\right)\left( x_{1}  - \frac{p_1}{p}\left(1 - \delta_n \right) \right) \right\}  u_n(x_{1})     \mathrm{~d} x_{1}\\
	& \times \int_{[0,1]^{n-2}}  \exp\left\{ -\left( \frac{2\beta}{x-1} -\frac{\beta}{\delta_n} \right) \sum_{i=1}^{n-2} \frac{px_i -p_1}{p_1}  \right\} \mathrm{~d} G_{n-2}^{p_1-2, p_2-2}.              
\end{align*}  

Note that $\frac{\beta(n-2)p}{p_1 \delta_n} + \frac{2\beta p}{p_1(x-1)} = O\left(\frac{\beta np}{p_1\delta_n} \right)=o(\beta p).$ Via similar argument as for \eqref{fn2}, one gets that  
\begin{align*}
\log\frac{{\rm I}_{4, n}}{(R_{n}^{p_1, p_2}(x))^2}\le \log\frac{A_{n-1}^{p_1-1, p_2-1}}{\beta p\delta_n}+\beta (n-2)\log\frac{p_1\delta_n}{p}+\log u_n\left(\frac{p_1(1-\delta_n)}{p}\right)+O(\beta n) & .
\end{align*} 
Therefore, it follows that 
\begin{align*}
	& \log\frac{{\rm I}_{4, n}}{(R_{n}^{p_1, p_2}(x))^2}\\
	\leq & \frac{\beta p_1}{2}\log \left(1-\delta_n \right) + \frac{\beta p_2}{2}\log \left(1+\frac{p_1}{p_2}\delta_n \right) + \frac{\beta(n-2)}{2} \log \frac{p_1 \delta_n^2 }{n} - \frac{1}{2} \log \left( \beta p_1  \delta_n^2\right) +O\left( \beta n \right) \\
\le &	- \frac{\beta p_1 \delta_n^2 }{4} + o\left( \beta p_1 \delta_n^2 \right).
\end{align*}  
This confirms \eqref{fn3} since $\beta n^2/p_1\ll \beta p_1\delta_n^2$ and then finishes the proof of the part in Theorem \ref{thmsimu} corresponding to {\bf H}$_0$. 
\subsection{The proof of the strong efficiency of $F_n$.}
For the strong efficiency of $F_n,$ we only need to show 
$$\varlimsup_{n\to\infty}\frac{\mathbb E^L\left[ F_n^2 \right]}{\mathbb P\left(\lambda_{(n)} > \frac{p_1}{p}x \right)^2} <+\infty.$$ 
Recall 
$${\bf H}_1:	\quad  \beta n p_1^2=O(p_2), \quad \frac{\log n}{n}\vee \frac{n}{p_1}\ll \beta  \quad \text{and}  \quad (\beta n)^5=O((\beta p_1)^3).$$ 
We first clarify how $\beta n p_1^2\ll p_2$ and $(\beta n)^5\ll (\beta p_1)^3$ serve  to make sure the validity of \eqref{ELFN2}.
Examining the verifications of \eqref{upperfor1}, \eqref{fn1}, \eqref{fn2} and \eqref{fn3}, the condition $\beta n p_1^2\ll p_2$ is used in Lemma \ref{lemexp1} so that \eqref{etalimit} holds, which will turn into the following limit when  $\beta n p_1^2=O(p_2)$ is considered \begin{align}
	\varlimsup_{n \rightarrow \infty} \mathbb{E}\left[\exp \left\{-\eta_1 \sum_{i=1}^n\left(\frac{p \lambda_i-p_1}{p_1}\right)-\eta_2 \sum_{i=1}^n\left(\frac{p \lambda_i-p_1}{p_1}\right)^2+\frac{\eta_2 n^2}{p_1}\right\}\right] <+\infty. \label{etalimit0}
\end{align}  
 When $(\beta n)^5=O((\beta p_1)^3),$ we choose specifically $\delta_n=4\sqrt{n/p_1}.$ First of all, we establish an alternative form of \eqref{alphalimit}. Indeed, for such $\delta_n,$ under the condition $\beta np_1^2=O(p_2), n\ll p_1$ and $(\beta n)^5=O((\beta p_1)^3),$ we have   
 $$
	\mathbb{E}\left[\exp \left\{-\alpha \sum_{i=1}^n\left(\frac{p \lambda_i-p_1}{p_1}\right)\right\}\right]=O(1) 
$$
for $\alpha=O(\beta/\delta_n)$ as $n$ large enough. We also see obviously that for such $\delta_n,$ it holds 
$$\frac{\beta n^2 \delta_n}{p_1}=\frac{4(\beta n)^{5/2}}{(\beta p_1)^{3/2}}=O(1).$$ 
We check one by one what it affects the whole proof of \eqref{ELFN2}. 
The lower bound of $\mathbb{P}(E_n)$ is modified to be
$$\mathbb P(E_n) 
	\geq R_n^{p_1, p_2}(x) \exp\left\{-\frac{\beta n^2}{2 p_1(x-1)^2}+O\left(1\right) \right\}.$$ Similarly, 
	 $${\rm I}_{1, n}
	\leq  (R_{n}^{p_1, p_2}(x))^2 \exp\left\{ -\frac{\beta  n^2}{p_1(x-1)^2}+O\left(1\right) \right\}.$$
The verification of \eqref{fn1} only utilized Lemma \ref{estintpxd}, whose conditions are still satisfied under {\bf H}$_1.$ The proofs of ${\rm I}_{3, n}$ and ${\rm I}_{4, n}$ are similar. We just explain how ${\rm I}_{3, n}$ is affected by the new choice of $\delta_n$ under {\bf H}$_1$. 
In fact, remaining the unchanged terms in the expression of ${\rm I}_{3, n},$ one gets  
\begin{align}
	{\rm I}_{3, n}
	\le & \frac{ 8 x^2 n^2(n-1)}{\beta^2 p^2 (x-1)^2} (A_n^{p_1, p_2})^2 (A_{n-1}^{p_1-1, p_2-1})\left(\frac{p_1(x-1)}{p} \right)^{2\beta(n-1)} \left(\frac{p_1\delta_n}{p}\right)^{\beta(n-2)}u_n^2\left(\frac{p_1 x}{p}\right)\notag\\
	& \times \int_{p_1(1+\delta_n)/p}^1 \exp\left\{ \left(\frac{\beta(n-2)p}{p_1 \delta_n} - \frac{2\beta p}{p_1(x-1)}\right)\left( x_{n-1}  - \frac{p_1}{p}\left(1+\delta_n \right) \right) \right\}  u_n(x_{n-1})     \mathrm{~d} x_{n-1}. \notag \\      
\end{align} 
Still taking $\kappa_n:=\frac{\beta(n-2)p}{p_1 \delta_n} -\frac{2\beta p}{p_1(x-1)},$ we know from $\delta_n=4\sqrt{n/p_1}$ that 
 $$\lim_{n\to\infty}\left(-\frac{\kappa_n}{\beta p\delta_n}+\frac{l}{2(1+\delta_n)}\right)=\lim_{n\to\infty}\left(-\frac{n}{p_1\delta_n^2}+\frac{1}{2(1+\delta_n)}\right)=\frac{7}{16}.$$ Hence, Lemma \ref{estintpxd} again  tells that 
 \begin{align*}
	 & \int_{p_1(1+\delta_n)/p}^1 \exp\left\{ \left(\frac{\beta(n-2)p}{p_1 \delta_n} - \frac{2\beta p}{p_1(x-1)}\right)\left( x_{n-1}  - \frac{p_1}{p}\left(1+\delta_n \right) \right) \right\}  u_n(x_{n-1})     \mathrm{~d} x_{n-1} \\
	 \leq & \frac{16(1+o(1))}{7\beta p \delta_n} u_n\left( \frac{p_1}{p}\left(1+\delta_n \right) \right).
\end{align*} 
Therefore, it follows that 
$$\log\frac{{\rm I}_{3, n}}{(R_{n}^{p_1, p_2}(x))^2}\le \log\frac{A_{n-1}^{p_1-1, p_2-1}}{\beta p\delta_n}+\beta (n-2)\log\frac{p_1\delta_n}{p}+\log u_n\left(\frac{p_1(1+\delta_n)}{p}\right)+O(\log n).$$ 
Considering the expression \eqref{An-1}, since $\log n\ll \beta n,$ and $\delta_n=4\sqrt{n/p_1},$ we have that  
\begin{align*}
&\log\frac{{\rm I}_{3, n}}{(R_{n}^{p_1, p_2}(x))^2}\\
\leq & \frac{\beta p_1}{2}\log \left(1+\delta_n \right) + \frac{\beta p_2}{2}\log \left(1-  \frac{p_1}{p_2}\delta_n \right) + \frac{\beta(n-2)}{2} \log \frac{p_1 \delta_n^2 }{n} - \frac{1}{2} \log \left( \beta p_1  \delta_n^2\right)+\frac{\beta n}{2} +o\left( \beta n \right) \\
=&	- \frac{\beta p_1 \delta_n^2 }{4} +\frac{\beta n}{2} \log \frac{p_1 \delta_n^2 }{n}+\frac{\beta n}{2}+o\left( \beta n \right)\\
=&-(\frac{7}{2}-2\log 2)\beta n+o(\beta n).
\end{align*} 
This implies immediately that 
\begin{align*}
&\log\frac{{\rm I}_{3, n}}{(R_{n}^{p_1, p_2}(x))^2}+\frac{\beta n^2 }{p_1(x-1)}\le-(\frac{7}{2}-2\log 2)\beta n+o(\beta n)\longrightarrow -\infty
\end{align*}
as $n\to+\infty.$ 
It confirms the validity of \eqref{fn2}. Taking account of all these facts, we get
$$\varlimsup_{n\to\infty}\frac{\mathbb{E} [F_n] }{\mathbb{P}^2(\lambda_{(n)}>p_1 x/p)}<+\infty,$$
which claims the strong efficiency of the IS estimator $F_n.$ 

\subsection{Proof of logarithmical efficiency of $F_n$.}
The logarithmical efficiency of $F_n$ has to be done in another way.  We only need to 
capture the dominated term of $\log \mathbb{E} [F_n]$ and $\log \mathbb{P}(\lambda_{(n)}>p_1 x/p).$ 
Recall 
$${\bf H}_2:	\quad   np_1\ll p_2, \quad \frac{\log n}{n}\vee \frac{n}{p_1}\ll \beta  \quad \text{and}   \quad n\ll p_1,$$ 
which is obviously more stringent than {\bf A}$_1.$ Therefore, the large deviation in \cite{MaLD} says that 
$$\log\mathbb{P}(\lambda_{(n)}>p_1 x/p)=-\frac{\beta p_1}{2}(x-1-\log x)+o(\beta p_1).$$
It remains to prove 
\begin{equation}\label{asymforF}
\log \mathbb{E}[F_n]=-\beta p_1(x-1-\log x)+o(\beta p_1).	
\end{equation}

For the asymptotic of $\log \mathbb{E} [F_n],$ we investigate the behavior of $\log R_{n}^{p_1, p_2}(x)$ under the assumption {\bf H}$_2.$ By definition, it follows that 
$$ \log R_{n}^{p_1, p_2}(x)=\log \frac{A_n^{p_1, p_2} }{\beta p}+\beta(n-1)\log\left(\frac{p_1(x-1)}{p}\right)+\log u_n\left( \frac{p_1 x}{p} \right) .
$$ 
 Lemma \ref{Ann} and careful calculus help us to write that 
 $$\aligned \log R_{n}^{p_1, p_2}(x)=-\frac{\beta p_1}{2}(x-1-\log x)+o(\beta p_1) 
 \endaligned $$ 
 when {\bf H}$_2$ holds. 
 
We still have $$\log \mathbb{E}[F_n]\le \log \left({\rm I}_{1, n}+{\rm I}_{2, n}+{\rm I}_{3, n}+{\rm I}_{4, n}\right).$$ 
For \eqref{ELFN2}, we work with efforts to dig out the term $$\exp\left\{-\frac{\beta n^2}{2 p_1(x-1)^2}+o\left(1\right) \right\}$$ in the expressions of the upper bounds of ${\rm I}_{i, n}$ with $i=1, 2, 3$ and $4$ and the term $o(1)$ is weakened to be $O(1)$ for the strong efficiency. Since now, what we are interested is to find out the term with order $\beta p_1,$ the term $-\frac{\beta n^2}{2 p_1(x-1)^2}+o\left(1\right)$ could be in further weakened to be $ O(\beta n^2/p_1),$ which will greatly reduce the calculus. In the proof of \eqref{alphalimit}, with $n\ll p_1$ and $ \log n\ll \beta n$ and  $n\ll \beta p_1,$ we obtain     
 $$\mathbb{E}\left[\exp \left\{-\eta_1 \sum_{i=1}^n\left(\frac{p \lambda_i-p_1}{p_1}\right)-\eta_2 \sum_{i=1}^n\left(\frac{p \lambda_i-p_1}{p_1}\right)^2+\frac{\eta_2 n^2}{p_1}\right\}\right]\le \exp\left\{O\left(\frac{\beta n p_1^2}{p}\right)\right\}$$ and then for the case when $np_1\ll p_2,$ we have  
 $$\mathbb{E}\left[\exp \left\{-\eta_1 \sum_{i=1}^n\left(\frac{p \lambda_i-p_1}{p_1}\right)-\eta_2 \sum_{i=1}^n\left(\frac{p \lambda_i-p_1}{p_1}\right)^2\right\}\right]=\exp\left\{o\left(\beta p_1\right)\right\}.$$
 Similarly, we take $\delta_n=4\sqrt{n/p_1}$ to get 
 $$
	\mathbb{E}\exp \left\{-\alpha \sum_{i=1}^n\left(\frac{p \lambda_i-p_1}{p_1}\right)\right\}=\exp\left\{o\left(\beta p_1\right)\right\} 
$$
 under {\bf H}$_2.$ 
 Since $\beta n^2 \delta_n=o(\beta p_1)$ once $n=o(p_1)$, it holds that under {\bf H}$_2$
$$\log\frac{{\rm I}_{1, n}}{(R_{n}^{p_1, p_2}(x))^2}\le o(\beta p_1), \quad \log\frac{{\rm I}_{2, n}}{(R_{n}^{p_1, p_2}(x))^2}\le -\frac{\beta p_1}{2}(x-1-\log x)+o(\beta p_1)$$
 and 
 $$\log\frac{{\rm I}_{3, n}}{(R_{n}^{p_1, p_2}(x))^2}\le -O^+(\beta n), \quad \log\frac{{\rm I}_{4, n}}{(R_{n}^{p_1, p_2}(x))^2}\le -O^+(\beta n).$$
 Here, $a_n=O^+(b_n)$  means $a_n>0$ and $a_n=O(b_n)$ and we use a similar symbol $a_n=o^+(b_n)$ to represent $a_n>0$ and $a_n=o(b_n).$ 
 Thus, 
 $$\mathbb{E}[F_n]\le \sum_{i=1}^4{\rm I}_{i, n}\le \exp\{o^+(\beta p_1)\},$$
which leads 
 \begin{equation}\label{asymforFl}\log \mathbb{E}[F_n]\le 2\log R_{n}^{p_1, p_2}(x)+o(\beta p_1)	
 \end{equation}
 and then with the help of the behavior of $\log R_{n}^{p_1, p_2}(x)$ one knows 
 $$\log \mathbb{E}[F_n]\le -\beta p_1(x-1-\log x)+o(\beta  p_1).$$ Next, we need to find the inverse direction. 
 It is obvious that 
 $$\log \mathbb{E}[F_n]\ge \log {\rm I}_{1, n}.$$ We are going to find the appropriate lower bound of ${\rm I}_{1, n}$ to match the upper bound. 
 Observing the form of $F_n,$ we get 
 $${\rm I}_{1, n}
	\ge  n^2 (A_n^{p_1, p_2})^2 (2r)^{-1}\int_{\Omega_{1, F}}  \prod_{i=1}^{n-1}\left( x_n - x_i \right)^{2\beta} u_n^2(x_n)  e^{r\left( x_n- p_1 x/p \right)} \mathrm{~d} x_n \mathrm{~d} G_{n-1}^{p_1-1, p_2-1}.$$
	Lemma \ref{interaction} entails in further as for $\mathbb{P}(E_n)$ that 
	 \begin{align}
	{\rm I}_{1, n}
	\geq & n^2 (A_n^{p_1, p_2})^2 r^{-1}\left(\frac{p_1 (x-1)}{p}\right)^{2\beta(n-1)} \int_{ p_1 x/p }^{p_1(x+\delta_n)/p}  u_n^2(x_n)  e^{r\left( x_n- p_1 x/p \right)}\mathrm{~d}x_{n}      \notag \\
	& \times \int_{\Omega_{\delta_n}} \exp\left\{ - \frac{2\beta}{x-1}  \sum_{i=1}^{n-1}\frac{px_i-p_1}{p_1} -  \frac{2\alpha_1 \beta}{(x-1)^2} \sum_{i=1}^{n-1}\left(\frac{px_i-p_1}{p_1}\right)^2   \right\} \mathrm{~d} G_{n-1}^{p_1-1, p_2-1} , \label{lowerforF}
\end{align}  
where $$\Omega_{\delta_n} :=  \left[\frac{p_1(1-\delta_n )}{p}, \; \frac{p_1(1+\delta_n )}{p} \right]^{n-1}, \quad \alpha_1=\frac{1}{2}+\frac{\delta_n}{x-1} . $$ 
We can claim similarly that under the assumption {\bf H}$_2$, the lower bound of \eqref{lowerforE} could be modified to be 
\begin{equation}\label{modilower}\int_{\Omega_{\delta_n}} \exp\left\{ - \eta_1\sum_{i=1}^{n-1}\frac{px_i-p_1}{p_1} -  \eta_2\sum_{i=1}^{n-1}\left(\frac{px_i-p_1}{p_1}\right)^2   \right\} \mathrm{~d} G_{n-1}^{p_1-1, p_2-1}=\exp\{o(\beta p_1)\}. 	
\end{equation}

  Applying Lemma \ref{estintpxd}  and \eqref{modilower} to the first and second integrals in \eqref{lowerforF}, respectively, one gets that 
  \begin{align*}
  \log {\rm I}_{1, n}&\ge 2\log \frac{(A_n^{p_1, p_2})}{\beta p}+2\log u_n\left(\frac{p_1 x}{p}\right)+2\beta(n-1)\log\frac{p_1(x-1)}{p}+o(\beta p_1)\\
  &=2\log R_{n}^{p_1, p_2}(x)+o(\beta p_1)\\
  &=-\beta p_1(x-1-\log x)+o(\beta p_1). 	
  \end{align*}
  Hence, with \eqref{asymforFl}, the desired expression \eqref{asymforF} is then obtained.  

%%%%%%%%%%%%%%%%%%%%%%%%%%%%%%%%%%%%%

\end{document}

\maketitle
This supplementary material provides the proofs of Lemmas $3 \sim  8$ in Section 3.

\section*{Proof of Technical lemmas}
For completeness, we first restate Lemma 1 and Lemma 2.

\begin{lemma}\label{loglem}
For $\alpha_1>1 / 2$ and $0<\alpha_2<1 / 2$, we have
\begin{align*}
	&\log (1-z) \geq-z-\alpha_1 z^2, \quad \text { if } z \in\left(-1,1-\frac{1}{2 \alpha_1}\right); \\
	&\log (1+z) \leq z-\alpha_2 z^2, \quad \text { if } z \in\left(-1, \frac{1}{2 \alpha_2}-1\right). 
\end{align*}  
\end{lemma}

\begin{lemma}\label{betalem}
Let $X$ be a Beta-distributed random variable with parameter $a$ and $b,$ denoted by $X\sim {\rm Beta}(a, b),$ that equivalently means the density function of $X$ is $$\frac{\Gamma(a+b)}{\Gamma(a)\Gamma(b)}x^{a-1} (1-x)^{b-1}1_{0<x<1}.$$ The following three assertions hold true.  
\begin{itemize}
	\item [(i)] The $k$-th moment of $X$ is
\begin{align*}
	\mathbb E \left[X^k\right]  = \prod_{i=1}^{k} \frac{ a +i-1}{a+b+i-1}.
\end{align*}  
	\item [(ii)] The variance of $X$ and $X^2$ are
\begin{align*}
\operatorname{Var}[X] & = \frac{a b}{(a+b)^2(a+b+1)}, \\
\operatorname{Var}\left[X^2\right] & =\frac{a(a+1)\left(4 a b(a+b+3)+6 b^2+b\right)}{(a+b)^2(a+b+1)^2(a+b+2)(a+b+3)}.
\end{align*}  
	\item[(iii)] For any $0 < \varepsilon < 1/2,$
$$
\mathbb P(|X- \mathbb E[X]|>\varepsilon) \leq 4 \exp \left\{-\frac{\varepsilon^2}{128} \cdot \frac{a^3+b^3}{a b}\right\}.
$$
\end{itemize}
\end{lemma}

\begin{lemma}\label{Ann}
Let  $A_n^{p_1, p_2}$ be defined by  
\begin{align}\label{defAn} 
A_n^{p_1, p_2} = \frac{\Gamma\left(1+\frac{\beta}{2}\right) \Gamma\left(\frac{\beta p}{2}\right) \Gamma\left(\frac{\beta(p-1)}{2}\right)}{\Gamma\left(1+\frac{\beta n}{2}\right) \Gamma\left(\frac{\beta p_1}{2}\right) \Gamma\left(\frac{\beta p_2}{2}\right) \Gamma\left(\frac{\beta(p-n)}{2}\right)}.
 \end{align}
Under the assumption $$
{\bf A}: \quad \quad \quad  \lim_{n\to \infty} \frac{n}{p_1} = 0, \quad  \lim_{n\to \infty} \frac{p_1}{p_2} = \sigma \in [0, \infty) \quad \text{ and } \quad  \beta n \!>\!>1,
$$ we have 
$$
	\log A_n^{p_1, p_2} =  \frac{\beta p_1}{2} \log \frac{p}{p_1} +\frac{\beta p_2}{2} \log \frac{p}{p_2}  +\frac{\beta (n-1)}{2} \log \frac{p}{n}  + \frac{1}{2}\log (\beta p_1) + \frac{1}{2}\beta n+o(\beta n). 
$$
\end{lemma} 

\begin{proof}
By definition,
\begin{align*} 
	\log A_{n}^{p_1, p_2}
	=& \log \Gamma\left(1+\frac{\beta}{2}\right)+\log \Gamma\left(\frac{\beta p}{2}\right)+\log \Gamma\left(\frac{\beta(p-1)}{2}\right) \\
	& -\log \Gamma\left(1+\frac{\beta n}{2}\right)-\log \Gamma\left(\frac{\beta p_1}{2}\right)-\log \Gamma\left(\frac{\beta p_2}{2}\right)-\log \Gamma\left(\frac{\beta(p-n)}{2}\right)  .
\end{align*}
The Stirling formula says that 
$$
\log \Gamma(x)=\left(x-\frac{1}{2}\right) \log x-x+ \frac{1}{2} \log\left(2 \pi\right)+o(1)
$$ 
for sufficiently large $x$. Then, by combining alike terms, we see that
\begin{align}
	\log A_{n}^{p_1, p_2}=
	&\log \Gamma\left(1+\frac{\beta}{2}\right)-\frac{\beta}{2} \log \frac{\beta}{2}+  \frac{\beta p_1}{2} \log \frac{p}{p_1} +\frac{\beta p_2}{2} \log \frac{p}{p_2} \notag \\ 
	& +  \left(\frac{\beta (p-n)}{2}-\frac{1}{2}\right) \log \frac{p-1}{p-n} +\frac{\beta(n-1)}{2} \log \frac{\beta(p-1)}{2+\beta n} \notag \\
	&+ \frac 12 \log \frac{\beta p_1}{2} - \frac{\beta+1}{2}\log \frac{n}{2}- \frac{1}{2} \log\left(2 \pi\right) +o(1) . \label{logAno1}
\end{align}
Since $\log \Gamma\left(1+\frac{\beta}{2}\right)-\frac{\beta}{2} \log \frac{\beta}{2}$ is either $O(\beta)$ for $\beta$ large enough or bounded for finite $\beta$ and $$  \left(\frac{\beta (p-n)}{2}-\frac{1}{2}\right) \log \frac{p-1}{p-n} = \frac{1}{2}\beta n+o(\beta n) ,$$
it comes 
\begin{align*} 
	\log A_{n}^{p_1, p_2} =  \frac{\beta p_1}{2} \log \frac{p}{p_1} +\frac{\beta p_2}{2} \log \frac{p}{p_2}  +\frac{\beta (n-1)}{2} \log \frac{p}{n} + \frac{1}{2}\log \beta p_1 + \frac{1}{2}\beta n+o(\beta n).
\end{align*}
\end{proof}

Now, we state a lemma on integrals of some strictly convex  function. 
\begin{lemma}\label{est1}
	Suppose $0 \le m< a<M \leq1$ and $k \geq 0$ and $b, c>1.$ Let $$\omega(t)= -k\,t -(b-1) \ln t-(c-1) \ln (1-t), \quad 0<t<1.$$
\begin{itemize}
\item[(i)] If $ \omega^{\prime}(a) >0,$ then we have 
	\begin{align*}
		\int_{a}^M e^{-\omega(t)} \mathrm{~d} t \leq \frac{e^{-k \,a} a^{b-1}(1-a)^{c-1} }{ \omega^{\prime}(a) } 	\end{align*}
and 
\begin{align*} 
	\int_a^M e^{-\omega(t)} \mathrm{~d} t  
	\geq & \frac{ e^{-k \,a} a^{b-1}(1-a)^{c-1}}{ \omega^{\prime}(M) }
			\left(1-  \exp \left\{ -(M-a)  \omega^{\prime}(a) \right\} \right). 
\end{align*}
\item[(ii)] By contrast,  it holds that 
	\begin{align*}
		\int_{m}^a e^{-\omega(t)} \mathrm{~d} t \leq  \frac{e^{-k \,a} a^{b-1}(1-a)^{c-1}}{ - \omega^{\prime}(a) }
	\end{align*}
and 
	\begin{align*}
		\int_{m}^a e^{-\omega(t)} \mathrm{~d} t \geq  \frac{e^{-k \,a} a^{b-1}(1-a)^{c-1}}{ - \omega^{\prime}(m) }\left(1-  \exp \left\{ (a-m) \omega^{\prime}(a) \right\} \right)
	\end{align*} 
	for $a$ when  $ \omega^{\prime}(a) <0.$
	\end{itemize}
\end{lemma}
\begin{proof}[Proof of Lemma \ref{est1}]
It is clear that $\omega^{\prime\prime}>0$ and then $\omega'$ is strictly increasing and has a unique root. For the case $ \omega^{\prime}(a)>0,$ which ensures $w'(x)>0$ once $x>a,$ we have
\begin{align*} 
 \int_{a}^M e^{-\omega(t)} \mathrm{~d} t
\leq & \int_a^M e^{-\omega(t)}\frac{\omega^{\prime}(t)}{\omega^{\prime}(a)} \mathrm{~d} t = \frac{e^{-\omega(a)}\left(1- e^{\omega(a)-\omega(M)}\right)}{\omega^{\prime}(a)}\le \frac{e^{-\omega(a)}}{\omega'(a)} 
 \end{align*}
and
\begin{align*} 
 \int_{a}^M e^{-\omega(t)} \mathrm{~d} t
\geq & \int_a^M e^{-\omega(t)}\frac{\omega^{\prime}(t)}{\omega^{\prime}(M)} \mathrm{~d} t = \frac{e^{-\omega(a)}\left(1- e^{\omega(a)-\omega(M)}\right)}{\omega^{\prime}(M)} .
 \end{align*}
 Then, the conclusion (i) follows from the facts
\begin{align*}
	(M-a)  \omega^{\prime}(a)
	\leq
	\omega(M)-\omega(a) 
	= \int_a^M \omega^{\prime}(t) \mathrm{~d} t 
	\leq  (M-a)  \omega^{\prime}(M).
\end{align*}
Similarly, for the case when $ \omega^{\prime}(a) <0,$ the conclusion holds since
\begin{align*}
	\frac{e^{-\omega(a)}\left(1- e^{\omega(a)-\omega(m)}\right)}{-\omega^{\prime}(m)} \leq  \int_m^{a} e^{-\omega(t)} \mathrm{~d} t 
	\leq  \frac{e^{-\omega(a)}\left(1- e^{\omega(a)-\omega(m)}\right)}{-\omega^{\prime}(a)}
\end{align*} 
and
\begin{align*}
	 (a-m)\omega^{\prime}(m) \leq \omega(a)-\omega(m) \leq (a-m)\omega^{\prime}(a) <0.
\end{align*}  
The proof is completed.
\end{proof}

Next, based on Lemma \ref{est1}, we provide accurate estimates on integrals of $u_n(x)=x^{r_{1, n}-1}(1-x)^{r_{2, n}-1}.$

\begin{lemma}\label{estintpxd}  
Suppose that $n<\!<\!p_1<\!<\!p_2.$ 
\begin{itemize}
\item[(i)] Given $x>1.$ For $z_n>0$ with $\frac n{p_1}<\!<\! z_n<\!<\! 1$ and $\kappa_n=o(\beta p),$ we have 
\begin{align*}
		 \int_{ \frac{p_1 x}{p} }^{\frac{p_1(x+z_n)}{p}} \exp\left\{\kappa_n \left(y-\frac{p_1 x}{p}\right)\right\} u_n\left(y\right) \mathrm{~d}y=
		\frac{2 x}{(x-1) \beta p }u_n\left(\frac{p_1 x}{p}\right)(1+o(1)).\end{align*}   
\item[(ii)] Given $z_n>0$ with  $\frac{n}{p_1}<\!<z_n\le M<\infty$  and let $l$ be a positive integer. Supposing that $d_n=O(\beta n)$ and  the sequence $\kappa_n$ satisfying $$\lim_{n\to\infty}\left(-\frac{\kappa_n}{\beta p z_n}+\frac{l}{2(1+z_n)}\right)=c>0,$$ we claim that 
\begin{align*}
	 &\quad \int_{\frac{p_1}{p}\left(1+z_n \right)}^{1}   \exp \left\{  \kappa_n\left(y-\frac{p_1(1+z_n)}{p}\right) \right\}  u_n^l(y) y^{d_n}\mathrm{~d} y \\
	  &	 \leq \frac{1}{c\, \beta p z_n} u_n^l\left(\frac{p_1(1+z_n)}{p}\right)\left(\frac{p_1(1+z_n)}{p}\right)^{d_n}(1+o(1)).
\end{align*}  
\item[(iii)] Given $z_n>0$ with $\frac{n}{p_1}<\!<z_n<\!<1$ and $\kappa_n=o(\beta p),$ we have 
\begin{align*}
	& \int_{ 0}^{ \frac{p_1(1-z_n)}{p} }  \exp \left\{ \kappa_n\left(y-\frac{p_1(1-z_n)}{p}\right) \right\}  u_n(y) \mathrm{~d} y
	\leq \frac{2}{\beta p z_n} u_n\left(\frac{p_1(1-z_n)}{p}\right)(1+o(1)). 
\end{align*}  
\end{itemize} 
\end{lemma}

\begin{proof}[Proof of Lemma \ref{estintpxd}]
Considering the function $\omega$ as 
$$\omega(t)= -\kappa_n t -\log u_n(t)$$ and  specifically taking  
$
a = p_1 x/p, M=p_1(x+z_n)/p,$
 we can figure out that  
\begin{align*}
		& \int_{ \frac{p_1 x}{p} }^{\frac{p_1(x+z_n)}{p}} \exp\left\{\kappa_n \left(y-\frac{p_1 x}{p}\right)\right\} u_n\left(y\right) \mathrm{~d}y =e^{-\kappa_n a}\int_{ \frac{p_1 x}{p} }^{\frac{p_1(x+z_n)}{p}} e^{-\omega (t)} dt.
\end{align*} 
Some algebras show that 
\begin{align*}
	\omega^{\prime}\left(a\right) &=-\kappa_n -\left(\frac{\beta (p_1-n+1)}{2}-1\right)\times \frac{p}{p_1 x}+\frac{\frac{\beta (p_2-n+1)}{2}-1}{1-\frac{p_1 x}{p}}.                     
\end{align*} 
Since $n<\!<\!p_1<\!<\!p_2$ and $\kappa_n=o(\beta p),$ the dominated term of $\omega^{\prime}(a)$ is $\beta p.$ Thus,  
$$\omega^{\prime}(a)=\frac{\beta p (x-1)}{2x} +   O\left(\frac{\beta n p}{p_1} \right)-\kappa_n=\frac{\beta p (x-1)}{2x}(1+o(1))$$
and similarly
$$\omega^{\prime}(M) =\frac{\beta p (x-1)}{2x}(1+o(1)).$$ 
These lead that 
$$
(M-a)\omega^{\prime}(a) =\frac{ (x-1) \beta p_1 z_n}{2x}(1+o(1)) \to \infty,
$$ 
and so is $(M-a)\omega^{\prime}(M). $
The upper and lower bounds in Lemma \ref{est1} coincide with each other and then  
\begin{align*}
	\int_{a}^M e^{-\omega(t)} dt=\frac{2 x}{(x-1) \beta p} e^{\kappa_n  a}u_n(a)(1+o(1)).
\end{align*} 
This verifies the first assertion. 
For the second one, we consider the function $\omega$ 
as $$\omega(t)=-\kappa_n t-l \log u_n(t)+d_n\log t$$ and take $a=\frac{p_1 (1+z_n)}{p}.$   
The behavior of  
$\omega'(a)$ is required again.  Simple calculus tells us that    
\begin{align*}
	\omega^{\prime}(a) &=-\kappa_n +\frac{l(r_{2, n}-1)}{1-a}-\frac{l (r_{1, n}-1)-d_n}{a}\\
	&=\beta p z_n\left(-\frac{\kappa_n}{\beta p z_n}+\frac{l }{2 (1+\delta_n)}\right)+O\left(\frac{\beta n p}{p_1}\right).                     
\end{align*} 
The condition $n/p_1\ll z_n$ and
the condition on $\kappa_n$ imply that 
$$\omega'(a)=c \beta p z_n(1+o(1)).$$
Hence, we get the desired upper bound. Similar argument works for the third item and is omitted here. 
This completes the proof.
\end{proof}
Recall that 
\begin{align}
	E_n = \left\{(x_1, \cdots, x_n) \in \Delta_n \left| \frac{p_1}{p}(1-\delta_n ) < x_1, x_{n-1} <  \frac{p_1}{p}(1+\delta_n ) ,  \frac{p_1}{p} x < x_n < \frac{p_1}{p} \left(x +\delta_n \right) \right. \right\}. \label{DefEn} 	
\end{align}
Here, $\Delta_n:=\{(x_1, \cdots, x_n)\in [0, 1]^n: x_1\le x_2\le \cdots\le x_n \}.$
\begin{lemma}\label{interaction} Given $x>1$ and $0<\delta_n<\!<\! 1$ and let $E_n$ be defined as above. Set $$\alpha_2 = \frac{1}{2} - \frac{\delta_n}{x-1}, \quad \alpha_1 = \frac{1}{2}+ \frac{\delta_n}{x-1} \quad \text{and} \quad \bar{x}_i=\frac{px_i-p_1}{p_1(x-1)}.$$ For all $(x_1, \ldots, x_n) \in E_n,$ the interaction $\prod_{i=1}^{n-1}\left(x_n-x_i\right)$ has lower and upper bounds as follows
\begin{align*}
	\prod_{i=1}^{n-1}\left(x_n-x_i\right) &\leq   \left(\frac{p_1(x-1)}{p}\right)^{n-1} \exp \left\{  (n+o(n))\frac{p x_n -p_1x}{p_1(x-1) } - \sum_{i=1}^{n-1}\bar{x}_i - \alpha_2 \sum_{i=1}^{n-1}\bar{x}_i^2  \right\};\\
	\prod_{i=1}^{n-1}\left(x_n-x_i\right) &\geq
	\left(\frac{p_1(x-1)}{p}\right)^{n-1} \exp  \left\{-\sum_{i=1}^{n-1} \bar{x}_i-\alpha_1 \sum_{i=1}^{n-1}\bar{x}_i^2\right\}.
\end{align*}  

	\end{lemma}
\begin{proof}[Proof of Lemma \ref{interaction}]
Set $$z_i= \frac{p x_n -p_1x}{p_1(x -1) } - \bar{x}_i$$ and we rewrite $\prod_{i=1}^{n-1} |x_n-x_i|$ as 
$$ 
	\prod_{i=1}^{n-1}\left(x_n-x_i\right) = \left(\frac{p_1 x-p_1}{p}\right)^{n-1} \exp\left\{ \sum_{i=1}^{n-1} \log (1+z_i) \right\}.
$$
Since $ -\frac{\delta_n}{x-1} \leq z_i \leq \frac{2\delta_n}{x-1} <  \frac{1}{2 \alpha_2}-1$ and $\alpha_2 < \frac{1}{2},$ it follows from Lemma \ref{loglem} that 
\begin{align*}
		 \sum_{i=1}^{n-1} \log (1+z_i) 
	 \leq &\sum_{i=1}^{n-1} (z_i - \alpha_2 z_i^2)  \\
	 \leq & (n+o(n))\frac{p x_n -p_1x}{p_1(x -1) } - \sum_{i=1}^{n-1}\bar{x}_i - \alpha_2 \sum_{i=1}^{n-1}\bar{x}_i^2, 
\end{align*}   
where in the last step we throw way the term $- \alpha_2(n-1)\left( \frac{p x_n - p_1}{p_1x-p_1}  \right)^2 $ and use the uniform bound $\bar{x}_i\leq \delta_n $  on $1 \leq i \leq n-1$ once $(x_1, \cdots, x_n)\in E_n.$  
The lower bound could be similarly obtained and its proof is omitted here.   
	
\end{proof}

\begin{lemma}\label{mainlem}
	Let $\left(\lambda_1, \lambda_2, \cdots, \lambda_n\right)$ be the $\beta$-Jacobi ensemble $\mathcal J_n(p_1, p_2)$. 
For $n$ large enough, we have 
	\begin{align*}
		&\mathbb E\left[ \sum_{i=1}^n\left( \frac{p\lambda_i-p_1}{p_1} \right) \right]=O\left(\frac{n^2}{p}\right);\\	
		&	
		\mathbb E \left[\sum_{i=1}^n\left( \frac{p\lambda_i-p_1}{p_1} \right)^2\right] = \frac{n^2}{p_1}+O\left(\frac{n^2}{p}+\frac{n}{\beta p_1}+\frac{n^2}{p_1^2}\right).
	\end{align*}
\end{lemma}

\begin{proof}[\bf Proof of Lemma \ref{mainlem}]
(i) 
	Let $c_{1}, \ldots, c_{n}, s_1, \ldots,  s_{n-1}$  be independent random variables distributed as
\begin{equation}\label{cksk}
	\begin{cases}
c_{k} \sim \operatorname{Beta}\left( \beta^{\prime} \left(p_1-k+1  \right) , \beta^{\prime} \left(p_2-k+1 \right) \right), &1\leq k \leq n ; \\
s_{k} \sim \operatorname{Beta}\left(\beta^{\prime} \left(n-k \right)  , \beta^{\prime} \left(p_1+p_2-n-k+1  \right)  \right), &1\leq k \leq n-1;
\end{cases}
\end{equation}
with $\beta'=\beta/2$ and define
$$
\begin{aligned}
&a_{k}=s_{k-1}\left(1-c_{k-1}\right)+c_k\left(1-s_{k-1}\right), \\
&b_{k}=\sqrt{c_k(1-c_k)s_k(1-s_{k-1})},
\end{aligned}
$$
with $c_{0}=s_{0}=0$. Consider the tridiagonal matrix
\begin{align}\label{trijn}
	\mathcal{J}_{n}=\left(\begin{array}{cccc}
a_{1} & b_{1} & & \\
b_{1} & a_{2} & \ddots & \\
& \ddots & \ddots & b_{n-1} \\
& & b_{n-1} & a_{n}
\end{array}\right).
\end{align}  
From Nagel (\cite{Nagel}),  the eigenvalues $\left(\lambda_1, \lambda_2, \cdots, \lambda_n\right)$ with density function \eqref{defjacobi} has the same distribution as the eigenvalues of the matrix $\mathcal{J}_{n}$.  
It follows from the elementary properties of a matrix and its eigenvalues that  
\begin{align} 
\sum_{i=1}^n\lambda_i &=\sum_{k=1}^n	[s_{k-1}\left(1-c_{k-1}\right)+c_k\left(1-s_{k-1}\right)]; \label{onemoment}\\
\sum_{i=1}^n\lambda_i^2&=\sum_{k=1}^n \left(s_{k-1}\left(1-c_{k-1}\right)+c_k\left(1-s_{k-1}\right)\right)^2+2\sum_{k=1}^{n-1}c_k(1-c_k)s_k(1-s_{k-1}). \label{twomoment}
\end{align}
Hence,
\begin{align*}
	\mathbb E \left[\sum_{i=1}^n\lambda_i\right]=\sum_{k=1}^{n-1} \left(  \mathbb E s_{k}  -\mathbb{E} c_k \mathbb E{s_k}-\mathbb{E} c_{k+1} \mathbb E{s_k}\right)+\sum_{k=1}^n\mathbb{E} c_k.
\end{align*}
Lemma \ref{betalem} tells us that 
\begin{align*}
	\mathbb E c_{k+1} 
%	= & \frac{p_1-k}{p-2k}= \frac{p_1-k}{p}+ \frac{2kp_1}{p^2}+ \frac{2k^2(p_1-p_2)}{p^3} + o\left( \frac{n^2}{p^2} \right), \\
	=  \frac{p_1-k}{p-2k}= \frac{p_1-k}{p}+  O\left( \frac{n p_1}{p^2}\right)
\text{ \quad and \quad }
	\mathbb E s_{k} 
%	= & \frac{n-k}{p-2k+1} = \frac{n-k}{p}+\frac{(n-k)(2k-1)}{p^2}+ o\left( \frac{n^2}{p^2} \right). \\
	=  \frac{n-k}{p-2k+1} = \frac{n-k}{p}+ O\left( \frac{n^2}{p^2} \right).
\end{align*} 
Thereby, we get 
\begin{equation}\label{summean} \aligned
	\mathbb E \left[\sum_{i=1}^n\lambda_i\right]=&\sum_{k=1}^{n} \left(  \mathbb E s_{k} + \mathbb E c_k -\mathbb{E} c_k \mathbb E{s_k}-\mathbb{E} c_{k+1} \mathbb E{s_k}\right)\\
	=& \sum_{k=1}^{n}\left(\frac{p_1+n-2k+1}{p}+O\left(\frac{p_1 n}{p^2}\right)\right)\\
	=& \frac{p_1 n}{p}+O\left(\frac{p_1 n^2}{p^2}\right).\endaligned \end{equation}
This brings directly that 
$$ \mathbb E \left[\sum_{i=1}^n\left(\frac{p \lambda_i-p_1}{p_1}\right)\right] =O\left(\frac{n^2}{p}\right).$$ 
For $\mathbb E \left[\sum_{i=1}^n\left( \frac{p\lambda_i-p_1}{p_1} \right)^2\right],$  
since  $$\aligned \sum_{i=1}^n\left( \frac{p\lambda_i-p_1}{p_1} \right)^2&=\frac{p^2}{p_1^2}\sum_{i=1}^n \lambda_i^2-\frac{2p}{p_1}\sum_{i=1}^n \lambda_i+n,
\endaligned $$
taking account of the expressions \eqref{summean}, we have that 
\begin{equation}\label{expsummean}\mathbb E \left[\sum_{i=1}^n\left( \frac{p\lambda_i-p_1}{p_1} \right)^2\right]=\frac{p^2}{p_1^2}\mathbb{E}\sum_{i=1}^n \lambda_i^2-n+O\left(\frac{n^2}{p}\right). \end{equation} 

Observing the right hand side of \eqref{twomoment} and picking up the dominated terms $\mathbb{E}(c_k^2)$ and $\mathbb{E}(s_k(c_{k+1}+c_k))$ from the expression of $\mathbb{E}\sum_{i=1}^n \lambda_i^2,$ we are able to write the following asymptotic  
\begin{equation}\label{squaremean} \mathbb{E}\sum_{i=1}^n \lambda_i^2=2\sum_{k=1}^{n-1}\mathbb{E}\left((c_{k+1}+c_k)s_k\right)+\sum_{i=1}^n\mathbb{E}c_k^2+O\left(\frac{n^2 p_1^2}{p^3}\right).\end{equation} 
It follows from Lemma \ref{betalem} that 
$$\aligned \mathbb{E} c_{k+1}^2&=\frac{(p_1-k)(\beta'(p_1-k)+1)}{(p-2k)(\beta'(p-2k)+1)}\\
&=\frac{(p_1-k)^2}{(p-2k)^2}+\frac{p_1-k}{p-2k}\left(\frac{\beta'(p_1-k)+1}{\beta'(p-2k)+1}-\frac{p_1-k}{p-2k}\right)\\
&=\frac{(p_1-k)^2}{p^2}+O\left(\frac{np_1^2}{p^3}+\frac{p_1}{\beta p^2 }\right)
\endaligned $$
uniformly for $0\le k\le n-1$ and similarly, 
$$\aligned \mathbb{E}(s_k(c_{k+1}+c_k))&=\frac{2(n-k)(p_1-k)}{p^2}+O\left(\frac{n^2 p_1}{p^3}+\frac{n}{p^2}\right). \endaligned $$
Putting these three formulas into the equality \eqref{squaremean}, we get 
$$\mathbb{E}\sum_{i=1}^n \lambda_i^2=\sum_{i=1}^n \frac{(p_1-k)^2}{p^2}+4\sum_{k=1}^{n-1}\frac{(n-k)(p_1-k)}{p^2}+O\left(\frac{n^2p_1^2}{p^3}+\frac{p_1 n}{\beta p^2}+\frac{n^2}{p^2}\right).$$ 
Simple calculus brings  
$$\mathbb{E}\sum_{i=1}^n \lambda_i^2=\frac{p_1^2n+p_1 n^2}{p^2}+O\left(\frac{n^2p_1^2}{p^3}+\frac{p_1 n}{\beta p^2}+\frac{n^2}{p^2}\right).$$ 
Plugging this expression into \eqref{expsummean}, we get 
$$\mathbb E \left[\sum_{i=1}^n\left( \frac{p\lambda_i-p_1}{p_1} \right)^2\right]=\frac{n^2}{p_1}+O\left(\frac{n^2}{p}+\frac{n}{\beta p_1}+\frac{n^2}{p_1^2}\right).$$
 This closes the proof. 
\end{proof}

\begin{lemma}\label{lemexp1}
Let $\left(\lambda_1, \lambda_2, \cdots, \lambda_n\right)$ be the $\beta$-Jacobi ensemble $\mathcal J_n(p_1, p_2)$ and suppose that $\beta np_1^2\ll p_2, \; n\ll p_1$ and $(\beta n)^3\ll (\beta p_1)^2.$  
We have the following three statements. 
\begin{enumerate} 
\item Given a positive sequence $\delta_n$ with $\sqrt{\frac{n}{p_1}}\ll \delta_n\ll 1$ and $\alpha=O\left( \frac{\beta}{\delta_n} \right).$ It holds that
\begin{align}
	\lim _{n \rightarrow \infty} \mathbb{E}\exp \left\{-\alpha \sum_{i=1}^n\left(\frac{p \lambda_i-p_1}{p_1}\right)\right\}= 1. \label{alphalimit}
\end{align}
\item For the case when $\eta_1=O(\beta)$ and $\eta_2=O(\beta),$ we have that 
\begin{align}
	\lim _{n \rightarrow \infty} \mathbb{E}\left[\exp \left\{-\eta_1 \sum_{i=1}^n\left(\frac{p \lambda_i-p_1}{p_1}\right)-\eta_2 \sum_{i=1}^n\left(\frac{p \lambda_i-p_1}{p_1}\right)^2+\frac{\eta_2 n^2}{p_1}\right\}\right] = 1. \label{etalimit}
\end{align}  
\item Given a positive sequence $\delta_n$ satisfying 
	$\sqrt{n/p_1}\ll \delta_n\ll 1$ and $0<\eta_1=O(\beta), 0<\eta_2=O(\beta).$ It holds that 
\begin{align}
	& \mathbb E\left[ \exp\left\{ -\eta_1\sum_{i=1}^n\left( \frac{p\lambda_i-p_1}{p_1} \right) - \eta_2 \sum_{i=1}^n\left( \frac{p\lambda_i-p_1}{p_1} \right)^2  \right\}  \mathbf 1_{ \left\{ \lambda_{(1)}> \frac{p_1}{p}(1-\delta_n ),  \lambda_{(n)} < \frac{p_1}{p}(1 + \delta_n ) \right\} } \right] \notag\\
	\ge  & \exp \left\{-\frac{\eta_2 n^2}{p_1}+o\left(1\right) \right\} \label{lowerforE}
\end{align}
for $n$ large enough. 
\end{enumerate}
\end{lemma}

\begin{proof}[\bf Proof of Lemma \ref{lemexp1}]

We first work on \eqref{etalimit}, which requires a precise upper bound of 
\begin{align}\label{kulw}
	 \mathbb E\left[ \exp\left\{ -\eta_1  \sum_{i=1}^n\frac{p\lambda_i}{p_1}- \eta_2  \sum_{i=1}^n\left( \frac{p\lambda_i-p_1}{p_1} \right)^2 \right\} \right] 
\end{align}  
for $\eta_i=O(\beta), i=1, 2.$ 
Consider the tridiagonal random matrix $\mathcal{J}_n$ in equation \eqref{trijn}. The relationships \eqref{onemoment} and \eqref{twomoment}  imply that
\begin{align*}
	{\rm J}_n:=& \frac{\eta_1 p}{p_1}\sum_{i=1}^n\lambda_i+ \eta_2 \sum_{i=1}^n\left( \frac{p\lambda_i-p_1}{p_1} \right)^2 \\
		= & \frac{\eta_1 p}{p_1}\sum_{k=1}^n \left( s_{k-1}\left(1-c_{k-1}-c_k\right)+c_k \right) + \frac{\eta_2 p^2}{p_1^2}  \sum_{k=1}^n  \left( s_{k-1}\left(1-c_{k-1}-c_k\right)+c_k - \frac{p_1}{p}  \right)^2 \\
	 & +\frac{2 \eta_2 p^2}{p_1^2}  \sum_{k=1}^{n-1}c_k(1-c_k)s_k(1-s_{k-1}).
\end{align*}  
We have then 
\begin{align*}
	{\rm J}_n\geq &	\frac{\eta_1 p}{p_1}\sum_{k=1}^n \left(s_{k}\left(1-c_{k}\right) -\frac{1}{2} s_k^2  -\frac{1}{2} c_k^2 +c_k\right) 
										+   \frac{2 \eta_2 p^2}{p_1^2}\sum_{k=1}^n \left(\left(c_k-c_k^2\right) s_k -\frac{1}{3}c_k^3 -\frac{2}{3}s_k^3  \right)
\end{align*} 
due to the non-negativity of quadratic term and the following inequalities
$$-s_{k-1} c_{k} \geq - \frac{1}{2} \left(s_{k-1}^2 +c_k^2\right), \quad - \left(c_k-c_k^2\right) s_k s_{k-1} \geq - c_k s_k s_{k-1} \geq - \frac{1}{3} \left(c_k^3+s_k^3+ s_{k-1}^3\right).$$
 
For simplicity, we set
\begin{align}
	f_1(x, y) : = &   \frac{\eta_1 p}{p_1}  \left( 1- y \right)
			    -  \frac{\eta_1 p}{2 p_1}x  	 
			    +  \frac{2 \eta_2 p^2}{p_1^2} \left(y-y^2\right)  \notag 
				 - \frac{4 \eta_2}{3}\frac{p^2}{p_1^2} x^2,  \\
	f_2( y) : = &    \frac{\eta_1 p}{p_1} -  \frac{\eta_1 p}{2 p_1} y  - \frac{2 \eta_2}{3}\frac{p^2}{p_1^2} y^2, \label{f2def}
\end{align}  
where $1 \leq k \leq n.$ 
By the independence of $c_k$ and $s_k$, here comes the following upper bound 
\begin{align}
	 \mathbb E\left[ \exp\left\{-\frac{\eta_1 p}{p_1}\sum_{i=1}^n\lambda_i-\eta_2 \sum_{i=1}^n\left( \frac{p\lambda_i-p_1}{p_1} \right)^2 \right\} \right]  
	\leq &   \prod_{k=1}^n   \mathbb E\left[  e^{-c_k f_2( c_k)- s_k f_1(s_k, c_k) } \right] \label{fsckk}
\end{align}   
with the convention $ f_1(s_n, c_n) =0.$ 

When $1\leq k \leq n-1$, we discuss the upper bound of $ \mathbb E\left[  e^{-f_2( c_k)-f_1(s_k, c_k)} \right]$ by dividing $c_k, s_k$ into different regions (as showed  in the figure below, with definitions provided later). 
\begin{center}

\begin{tikzpicture}[scale = 2]
\draw[-stealth,line width=0.2pt] (-0.2,0) -- (2.5,0);
\draw[-stealth,line width=0.2pt] (0,-0.2) -- (0,2.5);

\fill[gray!100!white] (0,0) rectangle (0.2, 0.3);	
\node[scale= 0.4] (s0) at (0.1,-0.1) {$\Omega_0$};
\node[scale= 0.4] (d0) at (-0.1,0.15) {$D_0$};
\draw[thin] (0.2,0) -- (0.2, 0.02);
\draw[thin] (0, 0.3) -- (0.02,0.3);

%%%%%%%%%%
\fill[gray!35] (0.2,0) rectangle (0.4, 0.3);	
\node[scale= 0.4] (s0) at (0.3,-0.1) {$\Omega_1$};
\draw[thin] (0.4,0) -- (0.4, 0.02);

\fill[gray!30] (0.2,0.3) rectangle (0.4, 0.6);	
\fill[gray!25] (0.2,0.6) rectangle (0.4, 0.9);	
\fill[gray!20] (0.2,0.9) rectangle (0.4, 2);	

%%%%%%%%%%

\fill[gray!30] (0.4,0) rectangle (0.6, 0.3);	
\node[scale= 0.4] (V) at (0.5,-0.1) {$\cdots$};
\draw[thin] (0.6,0) -- (0.6, 0.02);

\fill[gray!25] (0.4,0.3) rectangle (0.6, 0.6);	
\fill[gray!20] (0.4,0.6) rectangle (0.6, 0.9);	
\fill[gray!15] (0.4,0.9) rectangle (0.6, 2);	

%%%%%%%%% 

\fill[gray!25] (0.6,0) rectangle (0.8, 0.3);	
\node[scale= 0.4] (sj) at (0.7,-0.1) {$\Omega_{j_0}$};
\draw[thin] (0.8,0) -- (0.8, 0.02);

\fill[gray!20] (0.6,0.3) rectangle (0.8, 0.6);	
\fill[gray!15] (0.6,0.6) rectangle (0.8, 0.9);	
\fill[gray!12] (0.6,0.9) rectangle (0.8, 2);
%%%%%%%%

%%%%%%%%%%%%%%%
\fill[gray!18] (0.8,0) rectangle (2, 0.3);	
\node[scale= 0.4] (sj1) at (1.4,-0.1) {$\Omega_{j_0+1}$};
\draw[thin] (2,0) -- (2, 0.02);

\fill[gray!15 ] (0.8,0.3) rectangle (2, 0.6);	
\fill[gray!12] (0.8,0.6) rectangle (2, 0.9);	
\fill[gray!10] (0.8,0.9) rectangle (2, 2);

%%%%%%%%%%%%
\fill[gray!35] (0,0.3) rectangle (0.2, 0.6);
\node[scale= 0.4] (d1) at (-0.1,0.45) {$D_1$};
\draw[thin] (0, 0.6) -- (0.02,0.6);

\node[scale= 0.4] (d1) at (-0.1,0.6) {$\vdots$};

\fill[gray!30] (0,0.6) rectangle (0.2, 0.9);
\node[scale= 0.4] (d1) at (-0.1,0.75) {$D_{j_1}$};
\draw[thin] (0, 0.9) -- (0.02,0.9);
	
\fill[gray!25] (0,0.9) rectangle (0.2, 2);
\node[scale= 0.4] (d1) at (-0.1,1.45) {$D_{j_1+1}$};
\draw[thin] (0, 2) -- (0.02,2);

\node[scale= 0.4] (A) at (2,-0.1) {$1$};
\node[scale= 0.5] (C) at (2.4,-0.1) {$s_k$};

\node[scale= 0.4] (B) at (-0.1,2) {$1$};
\node[scale= 0.5] (S) at (-0.1,2.4) {$c_k$};

\end{tikzpicture}
 
\end{center}
We are going to show that 
$$\mathbb E\left[  e^{-c_k f_2( c_k)-s_k f_1(s_k, c_k)} \right]=\mathbb E\left[  e^{-c_k f_2( c_k)-s_k f_1(s_k, c_k)} \mathbf 1_{\Omega_0\times D_0}\right](1+o(1)).
$$
Precisely, let $M_n = p_1/n,$ $j_0 = \min \left\{j \in \mathbb N: j M_n \mathbb E s_k \geq \frac{1}{4} \right\}$ and 
\begin{align*}
	\Omega_j =  \begin{cases}
	\left[0,  M_n \mathbb E s_k  \right), & j=0;\\
	 \left[jM_n \mathbb E s_k, (j+1)M_n \mathbb E s_k  \right), & 1 \leq j \leq j_0;\\
	  \left[(j_0+1)M_n \mathbb E s_k, 1\right], & j=j_0+1
\end{cases}
\end{align*}  
and similarly define $j_1 = \min \left\{j \in \mathbb N: j N_n \mathbb E c_k \geq \frac{1}{4} \right\}$ with $N_n=\sqrt{\frac{p_1}{n} }$ and
\begin{align*}
	D_j =  \begin{cases}
	\left[0,  N_n \mathbb E c_k  \right), & j=0;\\
	 \left[jN_n \mathbb E c_k, (j+1)N_n \mathbb E c_k  \right), & 1 \leq j \leq j_1;\\
	  \left[(j_1+1)N_n \mathbb E c_k, 1\right], & j=j_1+1.
\end{cases}
\end{align*}  
What we will do is to treat the expectation in different regions $\Omega_i\times D_j,$ which requires different estimates on the function $$e^{-xf_1(x, y)-yf_2(y)} 1_{\Omega_i\times D_j}(x, y).$$    
We first investigate the expectation restricted to $\Omega_0\times D_0,$ which offers the main contribution to the whole expectation.

Note that $\mathbb{E} s_k=\frac{n-k}{p-2k+1}=O(\frac{n}{p})$ and $M_n \mathbb{E} s_k=O(\frac{p_1}{p}). $
When $ \left(x,y\right) \in \Omega_0 \times (0,1) ,$ we see $x\le M_n\mathbb E(s_k)$ and then \begin{align}
	f_1(x, y) 
%	= & \frac{p}{p_1}  x\left(1- y  \right)  -  \frac{p}{2 p_1}  x^2  	    +  2 \eta \frac{p^2}{p_1^2} \left(y-y^2\right) x   - \frac{4 \eta}{3}\frac{p^2}{p_1^2} x^3 \\
			    = &    \frac{\eta_1  p}{p_1} \left(1 - y\right) +  \frac{2 \eta_2 p^2}{p_1^2} \left(y-y^2\right)  +O \left( \eta_1\vee \eta_2\right).\label{f1omega0}
\end{align}  Therefore, it holds that  
$$\aligned \beta'(p-n-k+1)+f_1(x, y)& \ge \beta'(p-n-k+1)\left(1+\frac{f_1(x, y)}{\beta' p}\right)  \\
&=\beta'(p-n-k+1)\left(1+\frac{\eta_1(1-y)}{\beta' p_1}+\frac{2\eta_2 p(y-y^2)}{\beta' p_1^2}+O\left(\frac{\eta_1\vee \eta_2}{\beta p}\right)\right).
\endaligned $$
Furthermore, for $y\in D_0,$ we see that $0\le y\le O^+(p_1 N_n/p).$ Since both $\eta_1$ and  $\eta_2$ have the same order as $\beta$ and $1\ll N_n,$ it is true that 
$$-\frac{\eta_1 y}{\beta' p_1}-\frac{2\eta_2 py^2}{\beta' p_1^2}+O\left(\frac{\eta_1\vee \eta_2}{\beta p}\right)=O\left(\frac{N_n\vee N_n^2\vee 1}{p}\right)=O\left(\frac{N_n^2}{p}\right).$$
It follows that for $(x, y)\in \Omega_0\times D_0,$ 
$$\aligned \beta'(p-n-k+1)+f_1(x, y)& \ge \beta'(p-n-k+1)\left(1+\frac{\eta_1}{\beta' p_1}+\frac{2\eta_2 p y}{\beta' p_1^2}+O\left(\frac{N_n^2}{p}\right)\right).
\endaligned $$ 
As far as $f_2$ is concerned, it follows from the definition \eqref{f2def} that 
$$
f_2(y)= \frac{\eta_1 p}{p_1} -  \frac{\eta_1 p}{2 p_1} y - \frac{2\eta_2 }{3} \frac{p^2}{p_1^2}y^2 = \frac{\eta_1 p}{p_1} + O\left(\beta N_n^2   \right).
$$ Thus,
$$\beta'(p_2-k+1)+f_2(y)\ge \beta'(p_2-k+1)\left(1+\frac{\eta_1}{\beta' p_1}+O\left(\frac{N_n^2}{p}\right)\right)$$ 
for $y\in D_0.$
For simplicity, set 
$$\zeta_1=\frac{\eta_1}{\beta' p_1}, \quad \zeta_2=\frac{2\eta_2 p}{\beta' p_1^2}, \quad l_n=\frac{\eta_1}{\beta' p_1}+O\left(\frac{N_n^2}{p}\right) \quad \text{and}\quad \kappa_n=\frac{N_n^2}{p}.$$
The joint probability density function of $c_k$ and $s_k$ is 
$$p_k(x, y):=\frac{1}{B_{s_k} B_{c_k}} x^{\beta'(n-k)-1}(1-x)^{\beta'(p-n-k+1)-1}y^{\beta'(p_1-k+1)-1}(1-y)^{\beta'(p_2-k+1)-1}.$$
Here, $\beta'=\beta/2$ and 
$$B_{s_k} := \frac{\Gamma(\beta^{\prime}\left(n-k\right))\Gamma(\beta^{\prime}\left(p-n-k+1\right))}{ \Gamma(\beta^{\prime}\left(p-2k+1 \right))}, \quad B_{c_k}= \frac{\Gamma(\beta^{\prime}\left(p_1-k+1\right))\Gamma(\beta^{\prime}\left(p_2-k+1\right))}{\Gamma(\beta^{\prime}\left(p-2k+2 \right))}.$$
Using the elementary inequality $\log(1-x)\le x$ and the above upper bounds on $f_1(x, y)$ and $f_2(y)$ for $(x, y)\in \Omega_0\times D_0,$ we get that 
$$\aligned p^{(k)}(x, y) e^{-x f_{1}(x, y)-y f_2(y)}&\le \exp\left\{-\beta'(p-n-k+1)(1+\zeta_1+\zeta_2 y+O(\kappa_n))x\right\} \\
&\quad\times x^{\beta'(n-k)-1} \exp\left\{-\beta'(p_2-k+1) (1+l_n) y\right\}y^{\beta'(p_1-k+1)}.\endaligned $$
Thus, it holds that 
\begin{align}
&\quad \mathbb E\left[  e^{-c_k f_2( c_k)-s_k f_1(s_k, c_k)} \mathbf 1_{\Omega_0\times D_0}\right]\notag\\
&=\frac{1}{B_{s_k} B_{c_k}}\int_{\Omega_0\times D_0} e^{-xf_1(x, y)-yf_2(y)} p^{(k)}(x, y) dx dy\notag\\
&\le \frac{1}{B_{s_k} B_{c_k}}\int_{0}^{+\infty} \exp\left\{-\beta'(p_2-k+1) (1+l_n) y\right\}y^{\beta'(p_1-k+1)} dy \notag\\
&\quad \left(\int_{0}^{\infty}  \exp\left\{-\beta'(p-n-k+1)(1+\zeta_1+\zeta_2 y+O(\kappa_n))x\right\} x^{\beta'(n-k)-1} dx \right) \notag \\
&\le \frac{\Gamma(\beta'(n-k))}{B_{s_k} B_{c_k}} (\beta'(p-n-k+1))^{-\beta'(n-k)} e^{-\beta'(n-k) (\zeta_1+O(\kappa_n))} \notag\\
&\quad \times \int_{0}^{+\infty} y^{\beta'(p_1-k+1)} \exp\left\{-\beta'(p_2-k+1) (1+l_n) y-\beta'(n-k)\zeta_2 y\right\} dy \notag \\
&\le \frac{\Gamma(\beta'(n-k))\Gamma(\beta'(p_1-k+1))}{B_{s_k} B_{c_k}}(\beta'(p-n-k+1))^{-\beta'(n-k)} e^{-\beta'(n-k) (\zeta_1+O(\kappa_n))} \notag \\
&\quad\times (\beta'(p_2-k+1))^{-\beta'(p_1-k+1)}\exp\left\{-\beta'(p_1-k+1)\left(l_n+\frac{(n-k)\zeta_2}{p_2-k+1}\right)\right\}. \label{generalinte}
	\end{align}
Here, for the second inequality and the third equality, we use the definition of Gamma functions relating to $dx$ and then $dy,$ respectively and for the last inequality we use an additional inequaltiy $\log(1+x)\le x.$ 

The Stirling formula implies that
\begin{equation}\label{expressforomega2} 
\aligned
&\quad \frac{\Gamma\left(\beta^{\prime}(p-2 k+2)\right) }{\Gamma\left(\beta^{\prime}\left(p_2-k+1\right)\right)} \times \left( \beta^{\prime}\left(p_2-k+1\right)\right)^{-\beta^{\prime}\left(p_1-k+1\right)} \\
&= \left(\beta^{\prime}(p-2 k+2)\right)^{\beta^{\prime}(p-2 k+2)} \left(\beta^{\prime}(p_2- k+1)\right)^{-\beta^{\prime}(p_2-k+1)}\times \exp\left\{-\beta'(p_1-k+1)+O\left( \frac{p_1}{p}\right)\right\} \\
&=\exp\left\{-\beta'(p-2k+2)\log\left(1-\frac{p_1-k+1}{p-2k+2}\right)-\beta'(p_1-k+1)+O\left( \frac{p_1}{p}\right)\right\}  \\
&=\exp\left\{O\left(\frac{\beta p_1^2}{p}\right)\right\} 
\endaligned 
\end{equation}
and similarly 
\begin{align}
\frac{\Gamma\left(\beta^{\prime}(p-2 k+1)\right)}{\Gamma\left(\beta^{\prime}(p-n-k+1)\right)}\left(\beta^{\prime}(p-n-k+1)\right)^{-\beta^{\prime}(n-k)}=\exp\left\{O\left(\frac{\beta n^2 }{p}\right)\right\}. \label{expressforomega}
\end{align}
Here, for the last equality in \eqref{expressforomega2} we use the Taylor's formula $\log(1+x)=x+O(x^2)$ for $|x|$ small enough.

Putting \eqref{expressforomega2} and \eqref{expressforomega} back into the expression \eqref{generalinte}, we have 
\begin{align*}
&\quad \mathbb E\left[  e^{-c_k f_2( c_k)-s_k f_1(s_k, c_k)} \mathbf 1_{\Omega_0\times D_0}\right]\notag\\
&\le \exp\left\{-\beta'(n-k) \zeta_1+O\left(\frac{\beta p_1^2} p\vee \beta n\kappa_n\right)-\beta'(p_1-k+1)\left(l_n+\frac{ (n-k)\zeta_2}{p_2-k+1}\right)\right\}\notag \\
&=	\exp\left\{-\frac{\eta_1(p_1+n-2k+1)}{p_1}-\frac{2\eta_2 p(n-k)(p_1-k+1)}{p_1^2(p_2-k+1)}+O\left(\frac{\beta p_1^2\vee \beta n N_n^2 \vee \beta p_1N^2_n}{p}\right)\right\}. 
	\end{align*}
Using the condition $\beta n p_1^2\ll p_2$ and $(\beta n)^3=o((\beta p_1)^2)$, we can write  
$$\aligned  \frac{2\eta_2 p(n-k)(p_1-k+1)}{p_1^2(p_2-k+1)}&=\frac{2\eta_2 (n-k)}{p_1}+O\left(\frac{\beta n^2}{p_1^2}\right)=\frac{2\eta_2 (n-k)}{p_1}+o\left(\frac1n\right);\\
O\left(\frac{\beta p_1^2\vee \beta n N_n^2 \vee \beta p_1N^2_n}{p}\right)&=O\left(\frac{\beta p_1^2}{p}\right)=o\left(\frac1n\right).\endaligned $$
Finally, we get that 
\begin{align}
&\quad \mathbb E\left[  e^{-c_k f_2( c_k)-s_k f_1(s_k, c_k)} \mathbf 1_{\Omega_0\times D_0}\right] \le \exp\left\{-\frac{\eta_1(p_1+n-2k+1)}{p_1}-\frac{2\eta_2 (n-k)}{p_1}+o\left(\frac{1}{n}\right)\right\}. \label{key00}
	\end{align}
 
Now we work on $\Omega_0\times D_j$ for $1\le j\le j_1+1.$ Simply, it is valid that   
\begin{align}
	& \mathbb E\left[  e^{ -c_k f_2( c_k)-s_k f_1(s_k, c_k)} \mathbf 1_{\Omega_0 \times D_j} \left(s_k, c_k\right)  \right]
	\leq  \sup_{y \in D_j}  e^{-y f_2( y)} \mathbb{E} \left[  e^{-s_kf_1(s_k, c_k)} \mathbf 1_{\Omega_0\times D_j} \left(s_k, c_k\right)  \right]. \label{omegaj}
\end{align}  
Similarly as for \eqref{generalinte}, we work on 
the lower bound of $f_1(x, y)+\beta'(p-n-k+1)$ for $(x, y)\in \Omega_0\times D_j.$
Indeed, for $y\in D_j,$ we throw away the two positive terms of $f_1(x, y)$ in \eqref{f1omega0} to get  
$$\beta'(p-n-k+1)+f_1(x, y)\ge \beta'(p-n-k+1)(1+O\left(p^{-1}\right)).$$
When $1\le j\le j_1,$ as \eqref{generalinte}, letting $p_{2}^{(k)}$ denote the density function of $c_k,$ it holds that 
\begin{align}
&\quad \mathbb E\left[  e^{-s_k f_1(s_k, c_k)} \mathbf 1_{\Omega_0\times D_j}(s_k, c_k)\right]\notag\\
&\le \frac{1}{B_{s_k}}\iint_{\Omega_0\times D_j}e^{-[\beta'(p-n-k+1)(1+O\left(\frac1p\right)]x} x^{\beta'(n-k)-1} p_{2}^{(k)}(y) dx dy \notag \\
&\le \frac{\Gamma(\beta'(n-k))}{B_{s_k} } \left(\beta'(p-n-k+1)\left(1+O\left(p^{-1}\right)\right)\right)^{-\beta'(n-k)} \mathbb{P}(c_k\in D_j) \notag \\
&\le \frac{ \Gamma(\beta^{\prime}\left(p-2k+1 \right))}{\Gamma(\beta^{\prime}\left(p-n-k+1\right))}(\beta'(p-n-k+1))^{-\beta'(n-k)} e^{-O(\beta'n/p) } \mathbb{P} (c_k\ge jN_n\mathbb{E}c_k)\notag \\
&=\exp\left\{O\left(\frac{\beta n^2}{p}\right)\right\}\mathbb{P} (c_k\ge jN_n\mathbb{E}c_k). \label{general0j} 
	\end{align} 
	
Here, for the last but second inequality we use the expression \eqref{expressforomega}.
Lemma \ref{betalem} ensures that
\begin{align} 
	 \mathbb P\left( c_k \geq jN_n \mathbb E c_k \right) 
	 \leq &  \mathbb P\left( \left|c_k-\mathbb{E} c_k\right| \geq \left(jN_n -1 \right) \mathbb E c_k \right) \notag  \\
	 \leq & 4 \exp \left\{ - \frac{\left(jN_n -1 \right)^2 \left(\mathbb E c_k\right)^2}{128}  \frac{\beta^{\prime}(p_2-k+1)^2 }{p_1-k+1} \right\} \notag \\
	 \leq & 4 \exp \left\{ - \frac{\beta^{\prime} p_1}{512}\left( 1+o\left( 1\right) \right) \left(j+1\right)^2 N_n^2 \right\} \label{convexforck}
\end{align}  
since $2(jN_n-1) \geq (j+1)N_n$ and $\mathbb E c_k = \frac{p_1}{p}+ O\left(\frac{n}{p} \right).$
Further, for $y\in D_j,$ we know that 
$$y f_2(y)\ge - \frac{2\eta_2 }{3} \frac{p^2}{p_1^2}y^2\ge \frac{2\eta_2 }{3} \frac{p^2}{p_1^2} (j+1)^2N_n^2(\mathbb{E}c_k)^2\ge -\eta_2 (j+1)^2 N_n^2.$$
Due to the facts $\beta n^2 /p\ll 1\ll p_1N_n^2$ and $(j+1)^2\ge 3j,$ it follows from \eqref{omegaj}, \eqref{general0j} and \eqref{convexforck} that \begin{align*}
&\quad \mathbb E\left[  e^{-c_k f_2(c_k)-s_k f_1(s_k, c_k)} \mathbf 1_{\Omega_0\times D_j}\right]\le \exp \left\{ - \frac{\beta' p_1 (j+1)^2 N_n^2}{600}\right\}\le \exp \left\{ - \frac{\beta p_1 j N_n^2}{512}\right\}.
	\end{align*} 
	This implies immediately that 
\begin{align}
&\quad \mathbb E\left[  e^{-c_k f_2(c_k)-s_k f_1(s_k, c_k)} \mathbf 1_{\Omega_0\times \cup_{j=1}^{j_1} D_j}\right]\le \sum_{j=1}^{\infty} \exp \left\{ - \frac{\beta p_1 j N_n^2}{600}\right\}\le \exp \left\{ - \frac{\beta p_1 N_n^2}{1024}\right\}. \label{key0j} 
	\end{align}  
Similarly as \eqref{general0j}, when $\left(s_k, c_k\right) \in \Omega_0 \times D_{j_1+1},$ we have
\begin{align}
	  & \quad \mathbb E\left[  e^{-c_kf_2( c_k)-s_kf_1(s_k, c_k)} \mathbf 1_{\Omega_0 \times D_{j_1+1}} \left(s_k, c_k\right)  \right] \notag\\
	 &\leq e^{\frac{2\eta_2 p^2}{3p_1^2}}\mathbb E\left[  e^{-s_kf_1(s_k, c_k)} \mathbf 1_{\Omega_0 \times D_{j_1+1}} \left(s_k, c_k\right)  \right]\notag\\
	 &\leq  \exp\left\{\frac{2\eta_2 p^2}{3p_1^2}+O\left(\frac{\beta n^2}{p}\right)\right\}\mathbb{P}\left(|c_k-\mathbb{E}c_k|\ge \frac14\right)\notag\\
	 &\leq \exp\left\{-\frac{\beta' p^2}{2200 p_1}\right\} \label{ckdj1}
\end{align} 
for $n$ large enough.
Here, for the last inequality we use Lemma \ref{betalem} again to lead 
$$\mathbb{P}\left(|c_k-\mathbb{E}c_k|\ge \frac14\right)\le 4\exp\left\{-\frac{\beta'(p_2-k+1)^2}{2048(p_1-k+1)}\right\}\le \exp\left\{-\frac{\beta'p^2}{2100 p_1}\right\}$$ 
and we also utilize the fact that $\frac{2\eta_2 p^2}{3p_1^2}+O\left(\frac{\beta n^2}{p}\right)=o\left(\frac{\beta p^2}{p_1}\right).$  
Putting \eqref{key00}, \eqref{key0j} and \eqref{ckdj1} together, and keeping in mind $\beta p_1 N_n^2=\frac{\beta p_1^2}{n}\ll \frac{\beta p^2}{p_1},$ we are able to write that 
\begin{align}
	 & \mathbb E\left[  e^{-f_2( c_k)-f_1(s_k, c_k)} \mathbf 1_{\Omega_0 \times [0, 1]} \left(s_k, c_k\right)  \right] \notag \\
	 \leq & \exp\left\{  -\frac{\eta_1 (p_1+n-2k+1)}{p_1} -\frac{ 2 \eta_2 (n-k) }{p_1} +o\left( \frac{1}{n}\right)  \right\}+2\exp\left\{-\frac{ \beta p_1 N_n^2 }{1024}\right\}. \label{omega0}
\end{align}
 
 Now we work on $ \left(s_k, c_k\right) \in \Omega_j \times [0,1]$
 with $1\le j\le j_0.$ It is easy to write that
 \begin{align*}
	 \mathbb E \left[e^{ -s_k f_1(s_k, c_k)-c_k f_2(c_k) } \mathbf 1_{\Omega_j\times [0, 1]}(s_k, c_k) \right] & \leq  \sup_{x \in \Omega_j} e^{-x f_1(x, y) }\mathbb E \left[e^{ -c_k f_2(c_k)} \mathbf 1_{\Omega_j\times [0, 1]}(s_k, c_k) \right]\\
	 & \leq \sup_{x \in \Omega_j} e^{-x f_1(x, y) }\mathbb E \left[e^{ -c_k f_2(c_k)} \right]\mathbb P\left( s_k \geq jM_n \mathbb E s_k \right).
\end{align*}  
First, it follows from the definition
 in \eqref{f2def} and $x\le M_n(j+1)\mathbb{E}s_k$ that \begin{align*}
	 \sup_{x \in \Omega_j} e^{-x f_1(x, y) }
		 \leq  & \exp\left\{  \frac{ \eta_1 p}{2p_1}\left(j+1\right)^2 M_n^2  \left(\mathbb E s_k\right)^2 + \frac{4\eta_2 p^2}{3p_1^2}\left(j+1\right)^3 M_n^3  \left(\mathbb E s_k\right)^3 \right\} \\
   = & \exp \left\{O\left( \frac{\beta n^2 }{p_1^2} \right) \left(j+1\right)^2 M_n^2  \right\},
\end{align*}  
where the last equality is due to  
$$ \aligned  \frac{\eta_1 p}{2p_1}\left(\mathbb E s_k\right)^2 + \frac{4\eta_2 p^2}{3p_1^2}\left(j+1\right) M_n  \left(\mathbb E s_k\right)^3
\le \frac{\eta_1 p n^2}{2p_1 p^2} + \frac{4\eta_2 n^2 p^2}{3p_1^2 p^2}\left(j_0+1\right) M_n \mathbb{E} s_k=O\left( \frac{\beta n^2 }{p_1^2} \right).
\endaligned $$
Second, via Lemma \ref{betalem}, similarly we have   
\begin{align}
	 \mathbb P\left( s_k \geq jM_n \mathbb E s_k \right) 
	 \leq &  \mathbb P\left( \left|s_k-\mathbb{E} s_k\right| \geq \left(jM_n -1 \right) \mathbb E s_k \right) \notag \\
	 \leq & 4 \exp \left\{ - \frac{\left(jM_n -1 \right)^2 \left(\mathbb E s_k\right)^2}{128}  \frac{\beta^{\prime}(p-n-k+1)^2 }{n-k} \right\} \notag\\
	 \leq & 4 \exp \left\{ - \frac{\beta^{\prime}\left(n-k\right)\left(j+1\right)^2 M_n^2 }{512} \right\}. \label{convexsk}
\end{align}  
Thereby, we have that 
\begin{align*}
	 \quad \mathbb E \left[e^{ -s_k f_1(s_k, c_k)-c_k f_2(c_k) } \mathbf 1_{\Omega_j\times [0, 1]}(s_k, c_k) \right] &\leq  \exp \left\{ - \frac{\beta^{\prime}\left(n-k\right)\left(j+1\right)^2 M_n^2 }{600} \right\}\mathbb E \left[e^{ -c_k f_2(c_k)}\right]\\
	 &\le \exp \left\{ - \frac{\beta j M_n^2 }{600} \right\}\mathbb E \left[e^{ -c_k f_2(c_k)}\right] .\end{align*}
Utilizing the fact $s_k\le 1,$ we find that
\begin{align*}
	 & \mathbb E \left[e^{ -s_k f_1(s_k, c_k)-c_k f_2(c_k) } \mathbf 1_{\Omega_{j_0+1}\times [0, 1]}(s_k, c_k) \right]\\
	   \leq &\exp\left\{\frac{\eta_1 p }{2p_1}+\frac{4\eta_2p^2}{p_1^2} \right\}\mathbb E \left[e^{ -c_k f_2(c_k)}\right] \mathbb{P}\left(|s_k-\mathbb E s_k| \geq \frac14  \right)\\
		\leq & 4 \exp \left\{\frac{\eta_1 p}{2 p_1} +  \frac{4\eta_2 p^2}{3 p_1^2} -  \frac{\beta^{\prime} p^2}{2048 n} \left(1+ o(1) \right) \right\} \\
	\le  &  \exp \left\{ - \frac{\beta p^2}{5000 n} \right\}.
\end{align*}  
Thus, we have
\begin{align*}
	\mathbb{E}\left[e^{-c_kf_2\left(c_k\right)-s_kf_1\left(s_k, c_k\right)} \mathbf{1}_{\cup_{j=1}^{j_0+1} \Omega_j \times [0,1]}\left(s_k, c_k\right)\right] 
	\leq \exp\left\{ -\frac{\beta p_1^2}{1024 n^2}\right\} \mathbb E^{c_k} \left[ e^{-f_2( c_k)} \right].
\end{align*}  
Next, we focus on the upper bound of $ \mathbb E^{c_k} \left[ e^{- f_2( c_k)} \right]$ for $1 \leq k \leq n.$ Recalling $$\beta'(p_2-k+1)+f_2(y)\ge \beta'(p_2-k+1)\left(1+\frac{\eta_1}{\beta' p_1}+O\left(\frac{N_n^2}{p}\right)\right)$$ 
for $y\in D_0.$ 
Similar argument as \eqref{generalinte} and \eqref{key00}, we have from the definition of $c_k$ and \eqref{expressforomega2} that 
\begin{align*}
	 \mathbb E^{c_k} \left[ e^{-c_kf_2( c_k)} \mathbf{1}_{D_0}\left(c_k\right) \right]
	 \le & \exp\left\{-\beta'(p_1-k+1)\left(\frac{\eta_1}{\beta' p_1}+O\left(\frac{N_n^2}{p}\right)\right)+O\left(\frac{\beta p_1^2}{p}\right)\right\} \\
	= & \exp\left\{-\frac{\eta_1(p_1-k+1)}{p_1}+ O\left(\frac{\beta p_1^2}{p}\vee \frac{\beta p_1 N_n^2}{p}\right) \right\}\\
	&=\exp\left\{-\frac{\eta_1(p_1-k+1)}{p_1}+ O\left(\frac{\beta p_1^2}{p}\right) \right\}.
\end{align*}   
Also, as \eqref{key0j}, it holds that 
\begin{align*}
	\mathbb{E}\left[e^{-c_kf_2\left(c_k\right)} \mathbf{1}_{\cup_{j=1}^{j_1} D_j}\left(c_k\right)\right]
	\leq & \sum_{j=1}^{j_1} \sup _{y \in D_j} e^{-y f_2(y)} \mathbb{P}\left(c_k \geq j N_n \mathbb{E} c_k\right) \\
	\leq & \sum_{j=1}^{j_1} \exp \left\{\left(\eta_2-\frac{\beta^{\prime} p_1}{512} (1+o(1))\right)(j+1)^2 N_n^2\right\} \\
 	\leq & \exp \left\{-\frac{\beta p_1 N_n^2}{1024}\right\}
\end{align*}  
and similarly
\begin{align*}
	\mathbb{E}\left[e^{-c_kf_2\left(c_k\right)} \mathbf{1}_{ D_{j_1+1}}\left(c_k\right)\right] 
\leq \exp\left\{ - \frac{\beta p_2^2}{2048 p_1}  \right\} .
\end{align*}  
Combining the upper bounds provided above, we claim that 
\begin{align*}
	\mathbb E \left[ e^{-c_kf_2( c_k)} \right]\leq \exp\left\{-\frac{\eta_1(p_1-k+1)}{p_1}+ O\left(\frac{\beta p_1^2}{p}\right) \right\}+2 \exp \left\{-\frac{\beta p_1 N_n^2}{1024}\right\}
\end{align*}  
for all $1\le k\le n$ 
and therefore for $1 \leq k \leq n-1,$
\begin{align*}
	&\quad \mathbb{E}\left[e^{-c_kf_2\left(c_k\right)-s_kf_1\left(s_k, c_k\right)}\right]\\
	 &=\mathbb{E}\left[e^{- c_kf_2\left(c_k\right)-s_kf_1\left(s_k, c_k\right)} \mathbf{1}_{\cup_{j=1}^{j_0+1} \Omega_j \times [0,1]}\left(s_k, c_k\right)\right]+\mathbb{E}\left[e^{- c_kf_2\left(c_k\right)-s_kf_1\left(s_k, c_k\right)} \mathbf{1}_{\Omega_0 \times [0,1]}\left(s_k, c_k\right)\right]  \\
	& \le \exp\left\{  -\frac{\eta_1 (p_1+n-2k+1)}{p_1} -\frac{ 2 \eta_2 (n-k) }{p_1} +o\left( \frac{1}{n}\right)  \right\}+4\exp\left\{-\frac{ \beta p_1 N_n^2 }{1024}\right\}.
\end{align*}  
Substituting the upper bounds of $\mathbb{E}\left[e^{- c_kf_2\left(c_k\right)-s_kf_1\left(s_k, c_k\right)}\right]$ and $\mathbb E \left[ e^{ -f_2( c_k)} \right]$ into the inequality \eqref{fsckk} yields
\begin{align*}
	&\quad  \mathbb{E}\left[\exp \left\{-\eta_1 \sum_{i=1}^n\left(\frac{p \lambda_i-p_1}{p_1}\right)-\eta_2 \sum_{i=1}^n\left(\frac{p \lambda_i-p_1}{p_1}\right)^2\right\}\right] \\
	&\leq  \prod_{k=1}^{n-1}\bigg(\exp\left\{-\frac{\eta_1 (n-2k+1)}{p_1} -\frac{ 2 \eta_2 (n-k) }{p_1} +o\left( \frac{1}{n}\right)  \right\}+4\exp\left\{-\frac{ \beta p_1 N_n^2 }{1024}\right\}\bigg) \\
	&\quad \times \left(\exp\left\{\frac{\eta_1(n-1)}{p_1}+ O\left(\frac{N_n^2}{p}\right) \right\}+2 \exp \left\{-\frac{\beta p_1 N_n^2}{1024}\right\}\right) \\
	&= \prod_{k=1}^{n-1}\exp\left\{  -\frac{\eta_1 (n-2k+1)}{p_1} -\frac{ 2 \eta_2 (n-k) }{p_1} +o\left( \frac{1}{n}\right)  \right\}\left(1+\exp\left\{-O^+(\beta p_1 N_n^2) \right\}\right)\\
	&\quad \times \exp\left\{\frac{\eta_1(n-1)}{p_1}+ O\left(\frac{N_n^2}{p}\right) \right\}\left(1+\exp \left\{-O^{+}(\beta p_1 N_n^2)\right\}\right)\\
	&= \exp\left\{ -\frac{\eta_2 n^2}{p_1}  +o(1)\right\}\left(1+\exp \left\{-O^{+}(\beta p_1 N_n^2)\right\}\right).
		\end{align*}  
Here, for the second equality we use the fact that 
$$\beta p_1 N^2_n=\frac{\beta p_1^2}{n}>\!> \frac{\beta n}{p_1}.$$  
Therefore, it follows that \begin{align*}
	 \varlimsup_{n\to \infty}  \mathbb E\left[ \exp\left\{ -\eta_1 \sum_{i=1}^n\left( \frac{p\lambda_i-p_1}{p_1} \right) - \eta_2  \sum_{i=1}^n\left( \frac{p\lambda_i-p_1}{p_1} \right)^2  + \frac{\eta_2 n^2}{p_1} \right\} \right] 
	\leq & 1.	
\end{align*}  
On the other hand,  Jessen's inequality and Lemma \ref{betalem} implies the inverse direction 
\begin{align*}
	\varliminf_{n\to \infty}  \mathbb E\left[ \exp\left\{ -\eta_1 \sum_{i=1}^n\left( \frac{p\lambda_i-p_1}{p_1} \right) - \eta_2  \sum_{i=1}^n\left( \frac{p\lambda_i-p_1}{p_1} \right)^2  + \frac{\eta_2 n^2}{p_1} \right\} \right]
	\geq & 1.	
\end{align*}  
This finishes the proof of \eqref{etalimit}.

{\bf Proof of \eqref{alphalimit}}. It remains the limiting behavior of the following expectation 
\begin{align*}
	 \mathbb E\left[ \exp\left\{ -\alpha   \sum_{i=1}^n\frac{p\lambda_i}{p_1}\right\}\right], 
\end{align*} 
where $\alpha=O\left(\frac{\beta}{\delta_n}\right)$ with $\delta_n$ satisfying 
$\sqrt{\frac{n}{p_1}}\ll \delta_n\ll 1. $   
Setting 
\begin{align}
	f_1(x, y) : = \frac{\alpha p}{p_1}  \left( 1- y \right)
			    -  \frac{\alpha p}{2 p_1}x,   	 
			   \quad f_2( y) : = \frac{\alpha p}{p_1} -  \frac{\alpha  p}{2 p_1} y, \label{ftildedef}
\end{align}  
via the same analysis we have the following upper bound
\begin{align}
	 \mathbb E\left[ \exp\left\{-\frac{\alpha p}{p_1}\sum_{i=1}^n\lambda_i\right\}\right]  
	\leq &   \prod_{k=1}^n   \mathbb E\left[  e^{-c_k f_2( c_k)- s_k f_1(s_k, c_k) } \right] \label{tildefsckk}
\end{align}  
with $ f_1(s_n, c_n) =0.$ The argument will completely resemble the one for the second limit and we offer a sketch of the proof while keep a tense attention to the changes of the estimates of $f_1$ and $f_2$ while $\alpha$ replaces $\eta_1$ with different order and $\eta_2$ is replaced by $0.$ 

For $(x, y)\in \Omega_0\times D_0,$ replacing $\eta_2=0$ and paying attention to the order of $\alpha,$ it holds  that 
$$\aligned \beta'(p-n-k+1)+f_1(x, y)& \ge \beta'(p-n-k+1)\left(1+\frac{\alpha}{\beta' p_1}+O\left(\frac{N_n}{p\delta_n}\right)\right).
\endaligned $$
Similarly, 
\begin{align}\beta'(p_2-k+1)+f_2(y)\ge \beta'(p_2-k+1)\left(1+\frac{\alpha}{\beta' p_1}+O\left(\frac{N_n}{p\delta_n}\right)\right) 
\label{tildef20}\end{align} 
for $y\in D_0.$
Plugging all these expressions into \eqref{generalinte}, we have that 
\begin{align*}
\quad \mathbb E\left[  e^{-c_k f_2( c_k)-s_k f_1(s_k, c_k)} \mathbf 1_{\Omega_0\times D_0}\right]\le \exp\left\{-\frac{\alpha(p_1+n-2k+1)}{p_1}+O\left(\frac{\beta p_1^2}{p}\vee\frac{\beta p_1 N_n}{p\delta_n}\right)\right\}.
	\end{align*}
	As we compare the two terms in the $O,$ we see from the restriction $\delta_n^2\gg n/p_1$ that $$p_1^2\delta_n^2>\!>n p_1>\!>\frac{p_1}{n}=N_n^2.$$
	Then $O\left(\frac{\beta p_1^2}{p}\vee\frac{\beta p_1 N_n}{p\delta_n}\right)=O\left(\frac{\beta p_1^2}{p}\right)=o\left(\frac1n\right).$
	Thus, 
\begin{align}
\quad \mathbb E\left[  e^{-c_k f_2( c_k)-s_k f_1(s_k, c_k)} \mathbf 1_{\Omega_0\times D_0}\right]\le \exp\left\{-\frac{\alpha(p_1+n-2k+1)}{p_1}+o\left(\frac1n\right)\right\}. \label{key00tilde}
	\end{align}	
When $(x, y)\in \Omega_0\times D_j$ for $1\le j\le j_1+1,$
$$\beta'(p-n-k+1)+f_1(x, y)\ge \beta'(p-n-k+1)\left(1+O\left(\frac{1}{p\delta_n}\right)\right)$$
and $y f_2(y)\ge 0$
for $y\in D_j.$ 
When $1\le j\le j_1,$ the same analysis as \eqref{general0j} tells that 
\begin{align}
\quad \mathbb E\left[  e^{-c_kf_2(c_k)-s_k f_1(s_k, c_k)} \mathbf 1_{\Omega_0\times D_j}\right] &\le \mathbb E\left[  e^{-s_k f_1(s_k, c_k)} \mathbf 1_{\Omega_0\times D_j}\right] \notag\\
&\le 
\exp\left\{O\left(\frac{\beta n^2}{p}\vee \frac{\beta n}{p_1\delta_n}\right)\right\}\mathbb{P} (c_k\ge jN_n\mathbb{E}c_k)\notag\\
&\le 4 \exp \left\{ - \frac{\beta^{\prime} p_1}{512}\left( 1+o\left( 1\right) \right) \left(j+1\right)^2 N_n^2+O\left(\frac{\beta n^2}{p}\vee \frac{\beta n}{p_1\delta_n}\right) \right\}. \label{general0jtilde} 
	\end{align} 
Since $\frac{\beta n^2}{p}\vee \frac{\beta n}{p_1\delta_n}\ll \beta p_1N_n^2$ and $(j+1)^2\ge 3j,$ it is true that  \begin{align*}
&\quad \mathbb E\left[  e^{-c_k f_2(c_k)-s_k f_1(s_k, c_k)} \mathbf 1_{\Omega_0\times D_j}\right]\le \exp \left\{ - \frac{\beta p_1 j N_n^2}{512}\right\}.
	\end{align*} 
	Likewise, this implies immediately that 
\begin{align}
&\quad \mathbb E\left[  e^{-c_k f_2(c_k)-s_k f_1(s_k, c_k)} \mathbf 1_{\Omega_0\times \cup_{j=1}^{j_1} D_j}\right]\le \exp \left\{ - \frac{\beta p_1 N_n^2}{1024}\right\}. \label{key0jtilde} 
	\end{align}  
Similarly, we have
\begin{align}
	  \quad \mathbb E\left[  e^{-f_2( c_k)-f_1(s_k, c_k)} \mathbf 1_{\Omega_0 \times D_{j_1+1}} \left(s_k, c_k\right)  \right]
	 &\leq  \exp\left\{O\left(\frac{\beta n^2}{p}\vee \frac{\beta n}{p_1\delta_n}\right)\right\}\mathbb{P}\left(|c_k-\mathbb{E}c_k|\ge \frac14\right)\notag\\
	 &\leq \exp\left\{-\frac{\beta' p^2}{2200 p_1}\right\}. \label{ckdj1tilde} 
\end{align} 
Putting \eqref{key00tilde}, \eqref{key0jtilde} and \eqref{ckdj1tilde} together, we are able to write that 
\begin{align}
	 & \mathbb E\left[  e^{-f_2( c_k)-f_1(s_k, c_k)} \mathbf 1_{\Omega_0 \times [0, 1]} \left(s_k, c_k\right)  \right] \notag \\
	 \leq & \exp\left\{  -\frac{\alpha (p_1+n-2k+1)}{p_1}+o\left( \frac{1}{n}\right)  \right\}+2\exp\left\{-\frac{ \beta p_1 N_n^2 }{1024}\right\}. \label{omega0}
\end{align}
 Meanwhile, we can have 
 \begin{align*}
	 \sup_{x \in \Omega_j} e^{-x f_1(x, y) }
		 \leq  & \exp\left\{  \frac{ \alpha p}{2p_1}\left(j+1\right)^2 M_n^2  \left(\mathbb E s_k\right)^2 \right\} \\
   = & \exp \left\{O\left( \frac{\beta n^2 }{p_1 p\delta_n} \right) \left(j+1\right)^2 M_n^2  \right\}
\end{align*}   
Thereby, together with the concentration related to $s_k,$ we have that 
\begin{align*}
	 \quad \mathbb E \left[e^{ -s_k f_1(s_k, c_k)-c_k f_2(c_k) } \mathbf 1_{\Omega_j\times [0, 1]}(s_k, c_k) \right] &\leq  \exp \left\{ - \frac{\beta^{\prime}\left(n-k\right)\left(j+1\right)^2 M_n^2 }{600} \right\}\mathbb E \left[e^{ -c_k f_2(c_k)}\right]\\
	 &\le \exp \left\{ - \frac{\beta j M_n^2 }{600} \right\}\mathbb E \left[e^{ -c_k f_2(c_k)}\right] .\end{align*}
and  similarly 
\begin{align*}
	 & \mathbb E \left[e^{ -s_k f_1(s_k, c_k)-c_k f_2(c_k) } \mathbf 1_{\Omega_{j_0+1}\times [0, 1]}(s_k, c_k) \right]\\
		\leq & 4 \exp \left\{\frac{\alpha p}{2 p_1} -  \frac{\beta^{\prime} p^2}{2048 n} \left(1+ o(1) \right) \right\} \\
	\le  &  \exp \left\{ - \frac{\beta p^2}{5000 n} \right\}.
\end{align*}  
This is because $\frac{\alpha p}{p_1}=O\left(\frac{\beta p}{p_1\delta_n}\right)=o\left(\frac{\beta p^2}{n}\right).$
Thus, we have
\begin{align*}
	\mathbb{E}\left[e^{-c_kf_2\left(c_k\right)-s_kf_1\left(s_k, c_k\right)} \mathbf{1}_{\cup_{j=1}^{j_0+1} \Omega_j \times [0,1]}\left(s_k, c_k\right)\right] 
	\leq \exp\left\{ -\frac{\beta p_1^2}{1024 n^2}\right\} \mathbb E^{c_k} \left[ e^{-f_2( c_k)} \right].
\end{align*}  
Remembering  \eqref{f2def}, we have  
\begin{align*}
	 \mathbb E \left[ e^{-c_k f_2( c_k)} \mathbf{1}_{D_0}\left(c_k\right) \right]
	 \le & \exp\left\{-\beta'(p_1-k+1)\left(\frac{\alpha}{\beta' p_1}+O\left(\frac{N_n}{p\delta_n}\right)\right)+O\left(\frac{\beta p_1^2}{p}\right)\right\} \\
	= & \exp\left\{-\frac{\alpha(p_1-k+1)}{p_1}+ O\left(\frac{\beta p_1^2}{p}\vee \frac{\beta p_1 N_n}{p\delta_n}\right) \right\}\\
	=&\exp\left\{-\frac{\alpha(p_1-k+1)}{p_1}+ o\left(\frac{1}{n}\right) \right\}.
\end{align*} 
Here, we use the fact that 
$$\frac{N_n^2}{p_1^2\delta_n^2}=\frac{1}{n p_1\delta_n^2}\ll\frac{1}{n^2}\ll 1,$$
which makes $O\left(\frac{\beta p_1^2}{p}\vee \frac{\beta p_1 N_n}{p\delta_n}\right)=O\left(\frac{\beta p_1^2}{p}\right)=o\left(\frac1n\right).$   
Also, since $yf_2(y)\ge 0$ on $D_j$ when $1\le j\le j_1+1$, it holds that 
\begin{align*}
	\mathbb{E}\left[e^{-c_kf_2\left(c_k\right)} \mathbf{1}_{\cup_{j=1}^{j_1} D_j}\left(c_k\right)\right]
	\leq  \sum_{j=1}^{j_1} \mathbb{P}\left(c_k \geq j N_n \mathbb{E} c_k\right)
 	\leq \exp \left\{-\frac{\beta p_1 N_n^2}{1024}\right\}
\end{align*}  
and similarly
\begin{align*}
	\mathbb{E}\left[e^{-f_2\left(c_k\right)} \mathbf{1}_{ D_{j_1+1}}\left(c_k\right)\right] 
\leq \exp\left\{ - \frac{\beta p_2^2}{2048 p_1}  \right\} .
\end{align*}  
Combining the upper bounds provided above, we can tell that  
\begin{align*}
	\mathbb E \left[ e^{-c_k f_2( c_k)} \right]\leq \exp\left\{-\frac{\alpha(p_1-k+1)}{p_1}+ o\left(\frac{1}{n}\right) \right\}+2 \exp \left\{-\frac{\beta p_1 N_n^2}{1024}\right\}
\end{align*}  
for all $1\le k\le n$ 
and therefore for $1 \leq k \leq n-1,$
\begin{align*}
	&\quad \mathbb{E}\left[e^{-c_kf_2\left(c_k\right)-s_kf_1\left(s_k, c_k\right)}\right]\\
	 &=\mathbb{E}\left[e^{-c_kf_2\left(c_k\right)-s_kf_1\left(s_k, c_k\right)} \mathbf{1}_{\cup_{j=1}^{j_0+1} \Omega_j \times [0,1]}\left(s_k, c_k\right)\right]+\mathbb{E}\left[e^{-c_kf_2\left(c_k\right)-s_kf_1\left(s_k, c_k\right)} \mathbf{1}_{\Omega_0 \times [0,1]}\left(s_k, c_k\right)\right]  \\
	& \le \exp\left\{  -\frac{\alpha (p_1+n-2k+1)}{p_1} +o\left( \frac{1}{n}\right)  \right\}+4\exp\left\{-\frac{ \beta p_1 N_n^2 }{1024}\right\}.
\end{align*}  
Substituting the upper bounds of $\mathbb{E}\left[e^{-c_kf_2\left(c_k\right)-s_kf_1\left(s_k, c_k\right)} \right]$ and $\mathbb E^{c_k} \left[ e^{ -c_k'f_2( c_k)} \right]$ into the inequality \eqref{tildefsckk} yields
\begin{align*}
	& \mathbb{E}\left[\exp \left\{-\alpha \sum_{i=1}^n\left(\frac{p \lambda_i-p_1}{p_1}\right)\right\}\right] \\
	\leq & \prod_{k=1}^{n-1}\bigg(\exp\left\{-\frac{\alpha (n-2k+1)}{p_1} +o\left( \frac{1}{n}\right)  \right\}+4\exp\left\{-\frac{ \beta p_1 N_n^2 }{1024}\right\}\bigg) \\
	&\quad \times \left(\exp\left\{\frac{\alpha(n-1)}{p_1}+ o\left(\frac{1}{n}\right) \right\}+4\exp\left\{-\frac{ \beta p_1 N_n^2 }{1024}\right\}\right)\\
	&= 1+\exp \left\{-O^{+}(\beta p_1 N_n^2)\right\}.
		\end{align*}    
Therefore, together with the Jessen inequality and Lemma \ref{betalem} again, we obtain the desired result that  
\begin{align*}
	\lim_{n\to \infty}  \mathbb E\left[ \exp\left\{ -\alpha \sum_{i=1}^n\left( \frac{p\lambda_i-p_1}{p_1} \right)\right\} \right]
	= 1.	
\end{align*} 
The whole proof of \eqref{etalimit} is then completed.

{\bf Proof of \eqref{lowerforE}}. Using \eqref{etalimit}, we know $$\mathbb E\left[ \exp\left\{ -\eta_1  \sum_{i=1}^n\left( \frac{p\lambda_i-p_1}{p_1} \right) - \eta_2  \sum_{i=1}^n\left( \frac{p\lambda_i-p_1}{p_1} \right)^2  \right\} \right] = \exp \left\{-\frac{\eta_2 n^2}{p_1}+o\left(1\right) \right\}.$$  For the lower bound of 
$$\mathbb E\left[ \exp\left\{ -\eta_1  \sum_{i=1}^n\left( \frac{p\lambda_i-p_1}{p_1} \right) - \eta_2  \sum_{i=1}^n\left( \frac{p\lambda_i-p_1}{p_1} \right)^2  \right\} \mathbf 1_{\lambda_{(1)}>\frac{p_1(1-\delta_n)}{p}, \lambda_{(n)}<\frac{p_1(1-\delta_n)}{p}}\right],$$  
next we are going to show
\begin{align}
	 \mathbb E\left[ \exp\left\{ -\eta_1  \sum_{i=1}^n\left( \frac{p\lambda_i-p_1}{p_1} \right) - \eta_2  \sum_{i=1}^n\left( \frac{p\lambda_i-p_1}{p_1} \right)^2  \right\}\mathbf 1_{ \left\{ \lambda_{(n)} \ge \frac{p_1}{p}(1+\delta_n )\right\} } \right] 
	\le &\exp \left\{-\frac{\beta p_1 \delta_n^2}{4}(1+o(1))\right\}; \label{exp2n1n}\\
	 \mathbb E\left[ \exp\left\{ -\eta_1  \sum_{i=1}^n\left( \frac{p\lambda_i-p_1}{p_1} \right) - \eta_2  \sum_{i=1}^n\left( \frac{p\lambda_i-p_1}{p_1} \right)^2  \right\}\mathbf 1_{ \left\{ \lambda_{(1)} \le \frac{p_1}{p}(1-\delta_n )\right\} } \right] 
	\le  &\exp \left\{-\frac{\beta p_1 \delta_n^2}{4}(1+o(1))\right\}. \label{exp211n}
\end{align} 
Once these two inequalities hold, we see that 
\begin{align*}
	 &\quad \mathbb E\left[ \exp\left\{ -\eta_1  \sum_{i=1}^n\left( \frac{p\lambda_i-p_1}{p_1} \right) - \eta_2  \sum_{i=1}^n\left( \frac{p\lambda_i-p_1}{p_1} \right)^2  \right\} \mathbf 1_{\lambda_{(1)}>\frac{p_1(1-\delta_n)}{p}, \lambda_{(n)}<\frac{p_1(1-\delta_n)}{p}}\right]
\\	
	&=\mathbb E\left[ \exp\left\{ -\eta_1  \sum_{i=1}^n\left( \frac{p\lambda_i-p_1}{p_1} \right) - \eta_2  \sum_{i=1}^n\left( \frac{p\lambda_i-p_1}{p_1} \right)^2  \right\}  \right]\\
	&\quad-\mathbb E\left[ \exp\left\{ -\eta_1  \sum_{i=1}^n\left( \frac{p\lambda_i-p_1}{p_1} \right) - \eta_2  \sum_{i=1}^n\left( \frac{p\lambda_i-p_1}{p_1} \right)^2  \right\}\mathbf 1_{ \left\{ \lambda_{(1)} \le \frac{p_1}{p}(1-\delta_n )\right\} } \right] \\
	&\quad-\mathbb E\left[ \exp\left\{ -\eta_1  \sum_{i=1}^n\left( \frac{p\lambda_i-p_1}{p_1} \right) - \eta_2  \sum_{i=1}^n\left( \frac{p\lambda_i-p_1}{p_1} \right)^2  \right\}\mathbf 1_{ \left\{ \lambda_{(n)} \ge \frac{p_1}{p}(1+\delta_n )\right\} } \right] \\
	&\ge \exp \left\{-\frac{\eta_2 n^2}{p_1}+o\left(1\right) \right\}-2\exp \left\{-\frac{\beta p_1 \delta_n^2}{4}(1+o(1))\right\}. \label{exp211n}
\end{align*} 
The fact $\beta n^2/p_1\ll\beta p_1\delta_n^2$ entails the lower bound in \eqref{lowerforE}. 

We first prove the inequality \eqref{exp2n1n}. Note that on the set 
$\{\lambda_{(n)} > \frac{p_1}{p}(1+\delta_n )\},$  $\frac{p  \lambda_{(n)}-p_1}{p_1}>0,$ which implies with $\eta_1>0$ and $\eta_2>0$ that 
 $$ \exp\left\{ -\eta_1\left( \frac{p  \lambda_{(n)}-p_1}{p_1} \right) - \eta_2\sum_{i=1}^n\left( \frac{p  \lambda_{i}-p_1}{p_1} \right)^2  \right\}  \mathbf 1_{ \left\{ \lambda_{(n)} > \frac{p_1}{p}(1+\delta_n )\right\} } \le 1.$$
 Thus, utilizing the decomposition \eqref{orderden} and applying the inequality \eqref{upxn1xi} to see 
\begin{align*}
	&J_{1, n}:= \mathbb E\left[\exp\left\{ -\eta_1  \sum_{i=1}^n\left( \frac{p\lambda_i-p_1}{p_1} \right) - \eta_2  \sum_{i=1}^n\left( \frac{p\lambda_i-p_1}{p_1} \right)^2  \right\} \mathbf 1_{ \left\{ \lambda_{(n)} > \frac{p_1}{p}(1+\delta_n )\right\} } \right] \\
	\leq & \mathbb E\left[ \exp\left\{ -\eta_1  \sum_{i=1}^{n-1}\left( \frac{p\lambda_{(i)}-p_1}{p_1} \right) \right\}  \mathbf 1_{ \left\{ \lambda_{(n)} > \frac{p_1}{p}(1+\delta_n )\right\} } \right] \\
	\leq &  n A_n^{p_1, p_2} \left(\frac{p_1\delta_n}{p}\right)^{\beta(n-1)} \int_{\frac{p_1}{p}(1+\delta_n )}^1 \exp\left\{\frac{\beta (n-1) p}{p_1\delta_n} \left(x_{n}-\frac{p_1(1+\delta_n)}{p}\right)\right\} u_n(x_n) \mathrm{~d} x_n \\
	& \times \int  \exp\left\{ - \left( \eta_1 + \frac{\beta}{\delta_n} \right)  \sum_{i=1}^{n-1}\left( \frac{p x_i-p_1}{p_1} \right) \right\}\mathrm{~d} G_{n-1}^{p_1-1, p_2-1}.
\end{align*}  
Using Lemma \ref{estintpxd} and \eqref{alphalimit}, we have
\begin{align*}
	& \int_{\frac{p_1}{p}(1+\delta_n )}^1 \exp\left\{\frac{\beta (n-1) p}{p_1\delta_n} \left(x_{n}-\frac{p_1(1+\delta_n)}{p}\right)\right\} u_n(x_n) \mathrm{~d} x_n  \le \frac{2+o(1)}{\beta p \delta_n} u_n\left(\frac{p_1(1+\delta_n)}{p}\right)
\end{align*}  
and 
$$ \int  \exp\left\{ - \left( \eta_1 + \frac{\beta}{\delta_n} \right)  \sum_{i=1}^{n-1}\left( \frac{p x_i-p_1}{p_1} \right) \right\}\mathrm{~d} G_{n-1}^{p_1-1, p_2-1}=1+o(1)$$ 
for $n$ sufficiently large. 
Thus, it follows that 
\begin{align*}
	\log {\rm J}_{1, n}&\le \log \frac{A_{n}^{p_1, p_2}}{\beta p\delta_n}+\beta (n-1)\log\frac{p_1\delta_n}{p}+\log u_n\left(\frac{p_1(1+\delta_n)}{p}\right)+o(\beta n).
\end{align*}
Lemma \ref{Ann} and similar calculus as for \eqref{logI3nn} lead
\begin{align*}
	\log {\rm J}_{1, n}\le -\frac{\beta p_1 }{4} \delta_n^2 +o\left(\beta p_1\delta_n^2\right).
\end{align*}

Next, we prove the equality \eqref{exp211n}.  By the same argument as above, we have
\begin{align*}
	{\rm J}_{2, n}:=&\mathbb E\left[ \exp\left\{ -\eta_1\sum_{i=1}^n\left( \frac{p\lambda_i-p_1}{p_1} \right) - \eta_2  \sum_{i=1}^n\left( \frac{p\lambda_i-p_1}{p_1} \right)^2  \right\} \mathbf 1_{ \left\{ \lambda_{(1)}< \frac{p_1}{p}(1-\delta_n )\right\} } \right] \\
	\leq 	&  \mathbb E\left[ \exp\left\{ -\eta_1 \sum_{i=1}^n\left( \frac{p\lambda_{i}-p_1}{p_1} \right)\right\}\mathbf 1_{ \left\{ \lambda_{(1)}< \frac{p_1}{p}(1-\delta_n )\right\} } \right]. \end{align*} 
Similarly as for \eqref{upxn1xi}, one gets 
$$\prod_{i=2}^n (x_i-x_1)\le \left(\frac{p_1\delta_n}{p}\right)^{n-1}\exp\left\{\sum_{i=2}^n\frac{px_i-p_1}{p_1\delta_n}-\frac{px_1-p_1(1-\delta_n)}{p_1\delta_n}\right\}.$$	
This, together with the decomposition \eqref{orderden}, enables us to write 
\begin{align*}	 
	{\rm J}_{2, n}\leq &  n A_n^{p_1, p_2} e^{\eta_1\delta_n} \left(\frac{p_1}{p}\delta_n\right)^{\beta(n-1)} 
	\int  \exp\left\{ - \left( \eta_1 - \frac{\beta}{\delta_n} \right)  \sum_{j=2}^{n}\left( \frac{p x_i-p_1}{p_1} \right)\right\}\mathrm{~d} G_{n-1}^{p_1-1, p_2-1}\\
	& \times \int_0^{\frac{p_1}{p}(1 - \delta_n )} \exp\left\{-\frac{p(\beta (n-1) +\eta_1)}{p_1\delta_n}\left(x_{1}-\frac{p_1(1 -\delta_n )}{p} \right)\right\} u_n(x_1) \mathrm{~d} x_1 .
\end{align*}
The same applications of Lemma \ref{estintpxd}, Lemma \ref{lemexp1} and then Lemma \ref{Ann} deduce that 
$$
	\log {\rm J}_{2, n}
	\leq  -\frac{\beta p_1 \delta_n^2}{4}(1+o\left(1\right)).
$$
The proof of the inequality \eqref{exp211n} is then completed.
\end{proof}

%%%%%%%%%%%%%%%%%%%%%%%%%%%%%%%%%%%%%